\documentclass[11pt]{amsart}
\usepackage{amsmath,amssymb,amscd,amsthm,amsxtra}
\usepackage{latexsym}

\usepackage[dvips]{graphicx}        

\newtheorem{theorem}{Theorem}[section]
\newtheorem{lemma}{Lemma}[section]
\newtheorem{proposition}{Proposition}[section]

\newtheorem{definition}{Definition}[section]
\newtheorem{corollary}{Corollary}[section]

\newcommand{\e}{\varepsilon}

\newcommand{\p}{\partial}

\newcommand{\wh}{\widehat}

\newcommand{\ti}{\tilde}
\newcommand{\Id}{{\bf 1}}

\begin{document}
\title{Discrete Fourier Restriction associated with KdV equations}
\author{Yi Hu}
\address{
Yi Hu\\
Department of Mathematics\\
University of Illinois at Urbana-Champaign\\
Urbana, IL, 61801, USA}

\email{yihu1@illinois.edu}

\author{Xiaochun Li}

\address{
Xiaochun Li\\
Department of Mathematics\\
University of Illinois at Urbana-Champaign\\
Urbana, IL, 61801, USA}

\email{xcli@math.uiuc.edu}

\thanks{ This work was partially supported by an NSF grant DMS-0801154}

\begin{abstract} In this paper, we consider a discrete restriction associated with KdV equations. 
Some new Strichartz estimates are obtained.  We also establish the local well-posedness for 
 the periodic generalized Korteweg-de Vries equation with nonlinear term $ F(u)\p_x u$ provided $F\in C^5$
 and the initial data $\phi\in H^s$ with $s>1/2$.  
\end{abstract} 

\maketitle

\section{Introduction}
\setcounter{equation}0

The discrete restriction problem associated with KdV equations is a problem asking the best constant $A_{p, N}$ 
satisfying
\begin{equation}\label{res-1}
\sum_{n=-N}^{N} \left| \wh f(n, n^3)\right|^2\leq A_{p,N}\|f\|_{p'}^2\,,
\end{equation}
where $f$ is a periodic function on $\mathbb T^2 $, $\wh f$ is Fourier transform of $f$ on $\mathbb T^2$,
$p\geq 2$ and $p'=p/(p-1)$. It is natural to pose a conjecture asserting that for any $\e>0$, $A_{p, N}$ satisfies
\begin{equation}
A_{p,N}\leq
\begin{cases}
C_pN^{1- \frac{8}{p} + \e}\quad &\text{for}\ \  p\geq  8\\
C_p                            &\text{for}\ \  2\leq p< 8\,. 
\end{cases}
\end{equation}
It was proved by Bourgain that $A_{6,N}\leq N^\e$. The desired upper bound for $A_{8, N}$ is  
not yet obtained, however, we are able to establish an affirmative answer for large $p$ cases.

\begin{theorem}\label{thm1}
Let $A_{p, N}$ be defined as in (\ref{res-1}). If $p\geq 14$, then for any $\e>0$, there exists a
constant $C_p$ independent of $N$ such that
\begin{equation}\label{res-est}
A_{p, N} \leq C_p N^{1-\frac{8}{p}+ \e}\,. 
\end{equation}
\end{theorem} 

The periodic Strichartz inequality associated to KdV equations is the inequality seeking for the best constant $K_{p,N}$ satisfying 
\begin{equation}\label{Stri}
\left\|\sum_{n=-N}^N a_n e^{2\pi i t n^3+ 2\pi i x n }\right\|_{L^p_{x,t}(\mathbb T\times \mathbb T)}\leq K_{p, N} \left(
\sum_{n=-N}^N |a_n|^2\right)^{\frac{1}{2}}\,.
\end{equation} 
By duality, we see immediately 
$$
 K_{p, N}\sim \sqrt{A_{p,N}}\,.
$$
Henceforth, Theorem \ref{thm1} is equivalent to Strichartz estimates, 
\begin{equation}
K_{p, N}\leq CN^{\frac{1}{2}-\frac{4}{p}+\e}\,, \,\,\,{\rm for}\,\,\, {p\geq 14}\,. 
\end{equation}

It was observed by Bougain that the periodic Strichartz inequalities (\ref{Stri}) for $p=4, 6$
are crucial for obtaining the local well-posedness of periodic KdV (mKdV or gKdV).  
The local (global) well-posedness of periodic KdV for $s\geq 0$ was first studied by Bourgain in
\cite{B2}. Via a bilinear estimate approach, Kenig, Ponce and Vega in \cite {Ke} established the local 
well-posedness of periodic KdV for $s>-1/2$. 
The sharp global well-posedness of the periodic KdV was proved by Colliander,  Keel, Staffilani, Takaoka, and Tao in \cite{Tao1}, by utilizing the $I$-method.  \\

Inspired by Bourgain's work, we can obtain the following theorem on gKdV. 
Here the gKdV is the generalized  Korteweg-de Vries (gKdV) equation
\begin{equation}\label{gKdV}
\begin{cases}
u_t+u_{xxx}+u^ku_x=0\\
u(x,0)=\phi(x),\qquad x\in\mathbb{T},\ t\in\mathbb{R}\,,
\end{cases}
\end{equation}
where $k\in \mathbb N$ and $k\geq 3$.

\begin{theorem}\label{LWP-KdV}
The Cauchy problem (\ref{gKdV}) is locally well-posed if the initial data $\phi\in H^s$ for $s>1/2$.  
\end{theorem}

Theorem \ref{LWP-KdV} is not new. It was proved by Colliander,  Keel, Staffilani, Takaoka, and Tao 
in \cite{Tao}. However, our method is different from the method in \cite{Tao}. Let us point out 
the difference here. The method used in \cite{Tao} is based on a rescaling argument and the bilinear estimates, proved by Kenig, Ponce and Vega \cite{Ke}. Our method is more straightforward and does not 
need to go through the rescaling argument, the bilinear estimates in \cite{Ke} or the multilinear estimates in \cite{Tao}.
This allows us to extend Theorem \ref{LWP-KdV} to 
a very general setting. More precisely, consider the Cauchy problem for 
periodic generalized Korteweg-de Vries (gKdV) equation 
\begin{equation}\label{KdV0}
\begin{cases}
u_t+u_{xxx}+ F(u)u_x=0\\
u(x,0)=\phi(x),\qquad x\in\mathbb{T},\ t\in\mathbb{R}\,.
\end{cases}
\end{equation}
Here $F$ is a suitable function. Then the following theorem can be established.

\begin{theorem}\label{LWP-KdV-F}
The Cauchy problem (\ref{KdV0}) is locally well-posed provided 
$F$ is a $C^5$ function and the initial data $\phi\in H^s$ for $s>1/2$.  
\end{theorem}

For sufficiently smooth $F$, say $F\in C^{15}$, the existence of a local solution of (\ref{KdV0})
for $s\geq 1$ and the global well-posedness of (\ref{KdV0}) for small data $\phi\in H^s$ with $s>3/2$
were proved by Bourgain in \cite{B3}.   
The index $1/2$ is sharp because the ill-posedness of (\ref{gKdV}) for $s<1/2$ is known (see \cite{Tao}). 
In order to make (\ref{KdV0}) well-posed for 
the initial data $\phi\in H^s$ with $s>1/2$, 
the sharp regularity condition for $F$ perhaps is $C^4$. 
But the method utilized in this paper, with a small modification,  seems to be only able to reach an affirmative result for $F\in C^{\frac92+}$ and $s>1/2$. 
Moreover, the endpoint $s=1/2$ case could be possibly done by combining the ideas from \cite{Tao}
and this paper. But we would not pursue this endpoint result in this paper.  \\

\section{Proof of Theorem \ref{thm1}}
\setcounter{equation}0

To prove Theorem {\ref{thm1}}, we need to introduce a level set.
Since $\sqrt{A_{p,N}}\sim K_{p,N}$, it suffices to prove the Strichartz estimates (\ref{Stri}). 
Let $F_N $ be a periodic function on $\mathbb T^2$ given by
\begin{equation}\label{defofF0}
 F_N(x, t) = \sum_{n=-N}^N a_{n} e^{2\pi i nx} e^{2\pi i n^3 t}\,,
\end{equation}
where $\{a_n\}$ is a sequence with $\sum_n |a_n|^2 =1$ and $(x, t)\in \mathbb T^2$. For any 
$ \lambda>0 $, set a level set $E_\lambda$ to be
\begin{equation}\label{defofElam}
E_\lambda = \left\{ (x, t)\in \mathbb T^{2}: |F_N(x, t)|>\lambda \right\}\,.
\end{equation}
To obtain the desired estimate for the level set, let us first state a lemma 
on Weyl's sums. 

\begin{lemma}\label{lem1}
Suppose that $ t\in \mathbb T$ satisfies $|t-a/q|\leq 1/q^2$, where 
$a $ and $q$ are relatively prime. Then if $q\geq N^2$, 
\begin{equation}\label{cubic}
\left| \sum_{n=1}^N e^{2\pi i( tn^3 + bn^2 + cn)}\right|\leq 
CN^{\frac{1}{4}+\e}q^{\frac{1}{4}}\,. 
\end{equation}
Here $b$ and $c$ are real numbers, and the constant $C$ is independent of $b$, $c$, $t$, $a$, $q$ and $N$. 
\end{lemma}

The proof of Lemma \ref{lem1} relies on Weyl's squaring method. See \cite{Hua} or \cite{M} for detail. 
Also we need the following lemma proved in \cite{B1}.

\begin{lemma}\label{lem2}
For any integer $Q\geq 1$ and any integer $n\neq 0$, and any $\varepsilon>0$, 
$$\sum_{ Q\leq q < 2Q}\left|\sum_{a\in\mathcal{P}_q}e^{2\pi i\frac{a}{q}n}\right|\leq C_\varepsilon
 d(n, Q) Q^{1+\varepsilon}\,.$$
Here $\mathcal P_q $ is given by
\begin{equation}
 \mathcal P_q= \{a\in \mathbb N: 1\leq a\leq q \,\,\, {\rm and}\,\,\, (a,q)=1\}
\end{equation}
and $d(n, Q)$ denotes the number of divisors of $n$ less than $Q$ and $C_\varepsilon$ is a constant independent of
$Q, n$. 
\end{lemma}

Lemma \ref{lem2} can be proved by observing that the arithmetic function defined by 
$ f(q)=\sum_{a\in\mathcal{P}_q}e^{2\pi i\frac{a}{q}n}$ is multiplicative, and then utilize the prime 
factorization for $q$ to conclude the lemma.  \\

\begin{proposition}\label{Prop1}
Let $K_N$ be a kernel defined by
\begin{equation}\label{defofKN}
 K_{N}(x, t) =\sum_{n=-N}^N e^{2\pi i t n^3 + 2\pi i x n}\,. 
\end{equation} 
For any given positive number $Q$ with $N^2\leq Q\leq N^3$,
the kernel $K_N$  can be decomposed into $K_{1, Q} + K_{2, Q}$ such that 
\begin{equation}\label{K1}
\|K_{1, Q}\|_\infty \leq C_1 N^{\frac{1}{4}+\e}Q^{1/4}\,.
\end{equation}
and
\begin{equation}\label{K2}
\|\widehat{K_{2, Q}}\|_{\infty} \leq \frac{C_2 N^\varepsilon }{Q}\,.
\end{equation}
Here the constants $C_1, C_2$ are independent of $Q$ and $N$. 
\end{proposition}

\begin{proof}
We can assume that $Q$ is an integer, since  otherwise we can take the integer part of $Q$. 
For a standard bump function $\varphi$ supported on $[1/200, 1/100]$, we set 
\begin{equation}\label{defofPhi}
\Phi(t) = \sum_{ Q\leq q  \leq 5Q}\sum_{a\in\mathcal P_q}\varphi\left(\frac{t-a/q}{1/q^2}\right)\,.
\end{equation}
Clearly $\Phi$ is supported on $[0,1]$. We can extend $\Phi$ to other intervals periodically to obtain 
a periodic function on $\mathbb T$. For this periodic function generated by $\Phi$, we still use $\Phi$ to
denote it.  Then it is easy to see that
\begin{equation}
\widehat{\Phi}(0)=\sum_{q\sim Q}\sum_{a\in\mathcal{P}_q}\frac{\mathcal{F}_{\mathbb R}{\varphi}(0)}{q^2}=\sum_{q\sim Q}\frac{\phi(q)}{q^2}\mathcal{F}_{\mathbb R}{\varphi}(0) \,
\end{equation}
is a constant independent of $Q$. Here $\phi$ is Euler phi function, and $\mathcal{F}_{\mathbb R}$ denotes Fourier transform of a function on $\mathbb R$. Also we have
\begin{equation}\label{Fest}
\widehat{\Phi}(k)= \sum_{q\sim Q}\sum_{a\in\mathcal P_q}\frac{1}{q^2} e^{-2\pi i \frac{a}{q}k} \mathcal F_{\mathbb R}
\varphi(k/q^2)\,. 
\end{equation}
Applying Lemma \ref{lem2} and the fact that $Q\leq N^3$, we obtain
\begin{equation}\label{Fest1}
\left |\widehat{\Phi}(k)\right|\leq \frac{N^\e}{Q}\,, 
\end{equation}
if $k\neq 0$. 

We now define that 
$$
K_{1, Q}(x, t)=\frac{1}{\widehat{\Phi}(0)}K_N(x, t)\Phi(t), \,\,\,{\rm and}\,\,\, 
K_{2, Q} =  K_N - K_{1, Q}\,. 
$$

(\ref{K1}) follows immediately from Lemma \ref{lem1} since intervals $J_{a/q}=
[\frac{a}{q} + \frac{1}{100q^2}, \frac{a}{q} + \frac{1}{50q^2}]$'s  are pairwise disjoint
for all $Q\leq q\leq 5Q$ and $a\in\mathcal P_q$.

We now prove (\ref{K2}).  In fact, represent $\Phi$ as its Fourier series to get
$$
K_{2, Q}(x, t) = - \frac{1}{\widehat\Phi(0)} \sum_{k\neq 0}\widehat\Phi(k) e^{2\pi i k t} K_N(x, t)\,.
$$
Thus its Fourier coefficient is
$$
\widehat{K_{2, Q}}(n_1, n_2)= -\frac{1}{\widehat\Phi(0)} \sum_{k\neq 0} 
  \widehat\Phi(k){\bf 1}_{\{ n_2 = n_1^3 + k\} } (k) 
 \,.
$$
Here $(n_1, n_2)\in \mathbb Z^2$ and ${\bf 1}_A$ is the indicator function of a measurable set $A$.
This implies that 
$\widehat{K_{2, Q}}(n_1, n_2) =0$ if $ n_2 = n_1^3 $, and if $n_2\neq n_1^3$, 
$$  \widehat{K_{2, Q}}(n_1, n_2) =  -\frac{1}{\widehat\Phi(0)}
\widehat\Phi(n_2- n_1^3) \,.$$
Applying (\ref{Fest1}), we estimate $\widehat{K_{2, Q}}(n_1, n_2)$ by
$$
 \left|\widehat{K_{2, Q}}(n_1, n_2)\right| \leq \frac{CN^\varepsilon}{Q}\,,
$$
since $N^2\leq Q\leq N^3$. Henceforth we obtain (\ref{K2}). Therefore we complete the proof. 

\end{proof}

Now we can state our theorem on the level set estimates.  

\begin{theorem}\label{thm2}
For any positive numbers $\e$ and $Q\geq N^2$, the level set defined as in (\ref{defofElam}) satisfies 
\begin{equation}\label{estE}
\lambda^2 \left| E_\lambda\right|^2 \leq C_1 N^{\frac{1}{4}+\e}Q^{\frac{1}{4}}
\left| E_\lambda\right|^2 + \frac{C_2N^\e}{Q}\left| E_\lambda\right|\,
\end{equation}
for all $\lambda>0$.  Here $C_1$ and $C_2$ are constants independent of $N$ and $Q$.  
\end{theorem}

\begin{proof}
Notice that if $Q\geq N^3$, (\ref{estE}) becomes trivial since $E_\lambda=\emptyset$ if $\lambda
\geq CN^{1/2}$. So we can assume that $N^2\leq Q\leq N^3$. 
For the function $F_N$ and the level set $E_\lambda $ given in (\ref{defofF0}) and (\ref{defofElam}) respectively, we define $f$ to be
$$
f(x, t) = \frac{{F_N(x, t)}}{ |F_N(x, t)|} {\bf 1}_{E_\lambda}(x, t)\,. 
$$
Clearly
$$
\lambda|E_\lambda|\leq \int_{\mathbb T^2} \overline{F_N(x, t)} f(x, t) dx dt\,.  
$$
By the definition of $F_N$, we get
$$
\lambda|E_\lambda|\leq \sum_{n=-N}^N\overline{a_n} \widehat{f}(n, n^3 )\,.
$$
Utilizing Cauchy-Schwarz's inequality, we have
$$
\lambda^2|E_\lambda|^2\leq \sum_{n=-N}^N\left| \widehat{f}(n, n^3)\right|^2 \,.
$$
The right hand side can be written as
\begin{equation}
 \langle K_N *f, f\rangle\,. 
\end{equation}
For any $Q$ with $N^2\leq Q\leq N^3$, we employ Proposition \ref{Prop1} to decompose the kernel $K_N$. We then have 
\begin{equation}
\lambda^2|E_\lambda|^2\leq \left| \langle K_{1, Q} *f, f\rangle \right| 
   +  \left| \langle K_{2, Q} *f, f\rangle \right|  \,
\end{equation}
From (\ref{K1}) and (\ref{K2}), we then obtain 
$$
\lambda^2|E_\lambda|^2\leq C_1 N^{\frac{1}{4}+\e}Q^{\frac14}\|f\|_1^2 + \frac{C_2 N^\varepsilon}{Q}\|f\|_2^2
  \leq C_1 N^{\frac{1}{4}+\e}Q^{\frac14}|E_\lambda|^2 + \frac{C_2 N^\varepsilon}{Q}|E_\lambda|\,,
$$
as desired. Therefore, we finish the proof of Theorem \ref{thm2}. 
\end{proof}

\begin{corollary}\label{cor1}
If $\lambda\geq 2C_1N^{\frac{3}{8}+\e}$, then 
\begin{equation}\label{Eest2}
|E_\lambda|\leq \frac{CN^{1+\e}}{\lambda^{10}}\,. 
\end{equation}
Here $C_1$ is the constant $C_1$ in Theorem \ref{thm2} and $C$ is a constant independent of $N$ and $\lambda$.  
\end{corollary}

\begin{proof}
Since $\lambda\geq 2C_1N^{\frac{3}{8}+\e}$, we simply take $Q$ satisfies $ 2C_1N^{\frac{1}{4}+\e}Q^{1/4}= \lambda^2  $.
Then Corollary \ref{cor1} follows from Theorem \ref{thm2}. 
\end{proof}

We now are ready to finish the proof of Theorem \ref{thm1}. In fact, let $p\geq 14$ and 
write $\|F\|_p^p$ as
\begin{equation}\label{Fsplit}
 p\int_0^{2C_1N^{\frac{3}{8}+\e}}\lambda^{p-1} |E_\lambda|d\lambda  +
  p\int_{2C_1N^{\frac{3}{8}+\e}}^{2N^{1/2}}\lambda^{p-1} |E_\lambda|d\lambda  \,.  
\end{equation}
Observe that $ A_{6, N}\leq  N^\e $ implies 
\begin{equation}
 |E_\lambda|\leq \frac{N^\e}{\lambda^6}\,.  
\end{equation}
Thus the first term in (\ref{Fsplit}) is bounded by
\begin{equation}
  C N^{\frac{3(p-6)}{8} +\e } \leq CN^{\frac{p}{2}-4 +\e}\,,
\end{equation} 
since $p\geq 14$. 
From (\ref{Eest2}), the second term is majorized  by
\begin{equation}
 C N^{\frac{p}{2}-4+\e}\,.
\end{equation}
Putting both estimates together, we complete the proof of Theorem \ref{thm1}. \\

\section{A Lower bound of $A_{p, N}$}
\setcounter{equation}0

In this section we show that $N^{1-8/p}$ is the best upper bound of $A_{p, N}$ if $p\geq 8$.
Hence (\ref{res-est}) can not be improved substantially, and it is sharp up to a factor of $N^\e$. \\

For $b\in \mathbb N$, let $S(N; b)$ be defined by
\begin{equation}\label{defJNb}
S(N; b)=\int_{\mathbb T^2} \left| \sum_{n=-N}^N e^{2\pi i tn^3 + 2\pi i x n} \right|^{2b} dx dt\,.
\end{equation}

\begin{proposition}\label{prop2}
Let $S(N;b)$ be defined as in (\ref{defJNb}). Then
\begin{equation}\label{lbofJ}
S(N;b)\geq C \left(N^b + N^{2b-4}\right) \,.
\end{equation}
Here $C$ is a constant independent of $N$. 
\end{proposition}

\begin{proof}
Clearly $S(N;b)$ is equal to the number of solutions of 
\begin{equation}\label{sys}
\begin{cases}
 n_1+\cdots +n_b   =   m_1+\cdots + m_b \\
 n_1^3+\cdots + n_b^3 = m_1^3+\cdots +m_b^3  \,
\end{cases}
\end{equation}
with $n_j, m_j\in \{-N, \cdots, N\}$ for all $j\in\{1, \cdots, b\}$. 
For each $(m_1, \cdots, m_b)$, we may obtain a solution of (\ref{sys}) by taking 
$(n_1, \cdots, n_b)=(m_1, \cdots, m_b)$. Thus 
\begin{equation}\label{J-est}
S(N; b)\geq N^b\,. 
\end{equation}
To derive a further lower bound for $S(N; b)$, we set $\Omega$ to be
\begin{equation}\label{defofOme}
\Omega = \left\{(x, t):  |x|\leq \frac{1}{60 N}\,,\,\,\,\,  |t|\leq \frac{1}{60 N^3}\right\}\,. 
\end{equation}
If $(x, t)\in\Omega$ and $|n|\leq N$, then 
\begin{equation}
 \left| tn^3+xn\right|\leq \frac{1}{30}\,. 
\end{equation}
Henceforth if $(x, t)\in\Omega$, 
\begin{equation}
\left| \sum_{n=-N}^N e^{2\pi i tn^3 + 2\pi i x n} \right|\geq \left| {\rm Re}\sum_{n=-N}^N e^{2\pi i tn^3 + 2\pi i x n} \right| \geq \sum_{n=-N}^{N}\cos \left( 2\pi (tn^3+xn) \right)\geq CN\,.
\end{equation}
Consequently, we have 
\begin{equation}\label{Jest2}
S(N; b)\geq \int_{\Omega}\left| \sum_{n=-N}^N e^{2\pi i tn^3 + 2\pi i x n} \right|^{2b}dx dt\geq 
 CN^{2b}|\Omega|\geq CN^{2b-4}\,. 
\end{equation}
\end{proof}

\begin{proposition}\label{prop3}
Let $p\geq 2$ be even. Then $A_{p, N}$ satisfies 
\begin{equation}\label{LbofA}
  A_{p, N} \geq C(1+ N^{1-\frac{8}{p}})\,.
\end{equation}
Here $C$ is a constant independent of $N$. 
\end{proposition}

\begin{proof}
Let $p=2b$ since $p$ is even. Setting $a_n=1$ for all $n$ in the definition of $K_{p, N}$, we get
\begin{equation}
  S(N; b) \leq  K_{p, N}^p (2N)^{b}\,. 
\end{equation}
By Proposition {\ref{prop2}}, we have 
\begin{equation}
 K_{p, N} \geq C\left (1+ N^{\frac{1}{2}-\frac{4}{p}}\right)\,.
\end{equation}
Consequently, we conclude (\ref{LbofA}) since $A_{p,N}\sim K_{p,N}^2$. 
\end{proof}

\section{An estimate of Hua}
\setcounter{equation}0

The following theorem was proved by Hua in \cite{Hua} by an arithmetic argument. Here we utilize our method 
to provide a different proof.

\begin{theorem}\label{thmHua} 
Let $S(N; b)$ be defined as in (\ref{defJNb}). Then
\begin{equation}\label{Hua-est}
S(N;5)\leq CN^{6+\e}\,. 
\end{equation}
\end{theorem}

By Proposition \ref{prop2}, we see that the estimate (\ref{Hua-est}) is (almost) sharp.  
$S(N; 4)\leq N^{4+\e}$ is still open. We now prove Theorem \ref{thmHua}. 

\begin{proof}
Let $G_\lambda$ be the level set given by
\begin{equation}\label{defofG0}
 G_\lambda = \left\{ (x, t)\in \mathbb T^2: |K_N(x,t)|\geq \lambda  \right\}\,.
\end{equation}
Here $K_N$ is the function defined as in (\ref{defofKN}). \\

let $f=\Id_{G_\lambda}K_N/|K_N|$ and we then have
\begin{equation}\label{estGLam}
\lambda|G_\lambda| \leq \sum_{n=-N}^N \widehat f(n, n^3) = \langle f_{N}, K_N \rangle\,,
\end{equation} 
where $ f_N$ is a rectangular Fourier partial sum defined by 
\begin{equation}
f_N(x, t) = \sum_{\substack{|n_1|\leq N \\ |n_2|\leq N^3 } } \wh f(n_1, n_2) e^{2\pi n_1 x } e^{2\pi i n_2 t}\,.
\end{equation}

Employing Proposition \ref{Prop1} for $K_N$, we estimate
the level set $G_\lambda $ by
\begin{equation}
\lambda|G_\lambda|\leq |\langle f_{N}, K_{1, Q}\rangle| + |\langle f_{N}, K_{2, Q}\rangle |  \,
\end{equation}
for any $ Q\geq N^2$.
From (\ref{K1}) and (\ref{K2}),  $\lambda|G_\lambda| $ can be bounded further by
\begin{equation}
C\left( N^{\frac{1}{4}+\e}Q^{1/4}\|f_N\|_1 + \sum_{
\substack{ |n_1|\leq N\\ |n_2|\leq N^3   }  }
   \left| \wh{K_{2, Q}}(n_1, n_2)\wh f(n_1, n_2) \right|  \right)\,.
\end{equation}
Thus from the fact that $L^1$ norm of Dirichlet kernel $D_N$ is comparable to $\log N$, (\ref{K2}),  and Cauchy-Schwarz inequality, we have 
\begin{equation}
\lambda|G_\lambda|\leq C N^{\frac14+\e}Q^{1/4}|G_\lambda| + 
  \frac{C N^{2+\e}  }{Q}|G_\lambda|^{1/2}\,,
\end{equation}
for all $Q\geq N^2$. 
For $\lambda\geq  2C N^{\frac{3}{4} +\e}$, take $Q$ to be a number satisfying  $ 2CN^{\frac14+\e}Q^{1/4} = \lambda $ and then we obtain
\begin{equation}\label{estofG}
 |G_\lambda| \leq \frac{CN^{6+\e} }{\lambda^{10}}\,.
\end{equation}
Notice that 
\begin{equation}\label{L2ofS}
\|K_N\|_6 \leq   N^{\frac{1}{2}} K_{6, p}\leq N^{\frac{1}{2}+\e}\,. 
\end{equation}
Henceforth, by (\ref{estGLam}), we majorize $|G_\lambda|$ by 
\begin{equation}\label{estofG2}
|G_\lambda| \leq \frac{CN^{3+\e}}{\lambda^6} \,.
\end{equation}
We now estimate $S(N; 5)$ by
\begin{equation}\label{JN5est}
S(N;5)\leq C\int_{2CN^{\frac{3}{4}+\e}}^{2 N}
 \lambda^{10 -1 }|G_\lambda| d\lambda 
+ C\int_0^{2CN^{\frac34+\e}}
 \lambda^{10-1 }|G_\lambda| d\lambda  \,. 
\end{equation}
From (\ref{estofG}),  the first term in the right hand side 
of (\ref{JN5est}) can be bounded by $ CN^{6+\e}$. From (\ref{estofG2}), the second term is clearly bounded by
$N^{6+\e}$.  Putting both estimates together,
\begin{equation}\label{estofJnorm}
S(N; 5)\leq CN^{6+\e}\,,
\end{equation}
as desired. Therefore, we complete the proof.
\end{proof}

\section{Estimates for the nonlinear term and Local well-posedness of (\ref{gKdV})}\label{LWP1}
\setcounter{equation}0

For any measurable function $u$ on $\mathbb T\times \mathbb R$, we define the space-time Fourier transform by
\begin{equation}\label{defUhat}
\wh{u}(n, \lambda) = \int_\mathbb R \int_{\mathbb T}u(x, t) e^{- i n x} e^{- i \lambda t} dx \,dt\,
\end{equation} 
and set 
$$\langle x\rangle:= 1+|x|\,.$$

We now introduce the $X_{s, b}$ space, initially used by Bourgain.

\begin{definition}
Let $I$ be an time interval in $\mathbb R$ and $s, b\in\mathbb R$.  Let $X_{s, b}(I)$ be the space of functions $u$ 
on $ \mathbb T\times I $ that may be represented as
\begin{equation}
u(x,t) = \sum_{n\in \mathbb Z}\int_{\mathbb R} \wh u(n, \lambda) e^{ i nx} e^{ i \lambda t} d\lambda\,\,\,
{\rm for}\,\,\,\, (x, t)\in \mathbb T\times I\,
\end{equation}
with the space-time Fourier transform $\wh u$ satisfying
\begin{equation}
\|u\|_{X_{s,b}(I)} =  \left ( \sum_n\int\langle n\rangle^{2s}\langle \lambda-n^3\rangle^{2b}|\widehat{u}(n, \lambda)|^2d\lambda  \right)^{1/2} <\infty\, .
\end{equation}
Here the norm should be understood as a restriction norm. 
\end{definition}

We should take the time interval to be $[0, \delta]$ for a small positive number $\delta$, 
and abbreviate $\|u\|_{X_{s, b}(I)}$ as $\|u\|_{s, b}$ for any function $u$ restricted to 
$ \mathbb T\times [0, \delta]$.   In this section, we always restrict the function $u$ to $\mathbb T\times [0, \delta]$. 
Let $w$ be the nonlinear function defined by
\begin{equation}\label{defofw}
w = \left( u^k - \int u^k dx\right)u_x\,.
\end{equation}
  We also
define
\begin{equation}\label{defofNorm}
\|u\|_{Y_s}:= \|u\|_{s, \frac{1}{2}} +
\left(\sum_n\langle n\rangle^{2s}\left(\int\left|\widehat{u}(n,\lambda)\right|d\lambda\right)^2\right)^\frac{1}{2}\,.
\end{equation}

We need the following estimate 
on the nonlinear function $w$, in order to establish a contraction on the space $\{u: \|u\|_{Y_s}\leq M\}$ for some 
$M>0$.

\begin{proposition}\label{propofw}
For $s>1/2$, there exists $\theta>0$ such that, for the nonlinear function $w$ 
 given by (\ref{defofw}),  
\begin{equation}
\|w\|_{s, -\frac12} + \left(\sum_n\langle n\rangle^{2s}\left(\int\frac{|\widehat{w}(n,\lambda)|}{\langle\lambda-n^3\rangle}d\lambda\right)^2\right)^\frac{1}{2} \leq C\delta^\theta \|u\|_{Y_s}^{k+1}.
\end{equation}
Here $C$ is a constant independent of $\delta$ and $u$. 
\end{proposition}

The proof of Proposition \ref{propofw} will appear in Section \ref{proofP1}. 
We now start to derive the local well-posedness of (\ref{gKdV}).  For this purpose, we only need to 
consider the well-posedness of the Cauchy problem:
\begin{equation}\label{KdV2}
\begin{cases}
u_t+u_{xxx}+\left(u^k-\int_\mathbb{T}u^kdx\right)u_x=0\\
u(x,0)=\phi(x),\qquad x\in\mathbb{T},\ t\in\mathbb{R}
\end{cases}.
\end{equation}

This is because if $v$ is a solution of (\ref{KdV2}), then the gauge transform 
\begin{equation}\label{Gauge1}
u(x,t):= v\left(x-\int_0^t\int_\mathbb{T}v^k(y,\tau)dyd\tau,  t  \right)\,.
\end{equation}
is a solution of (\ref{gKdV}) with the same initial value $\phi$. Notice that this 
transform is invertible and preserves the initial data $\phi$. The inverse transform is
\begin{equation}
v(x,t):= u\left(x+\int_0^t\int_\mathbb{T}u^k(y,\tau)dyd\tau,  t  \right)\,.
\end{equation} 
It is easy to see that for any solution $u$ of (\ref{gKdV}), this inverse transform 
of $u$ defines a solution of (\ref{KdV2}).   
Hence to establish well-posedness of (\ref{gKdV}), it suffices to obtain the well-posedness of 
(\ref{KdV2}). This gauge transform was used in \cite{Tao}. \\

By Duhamel principle, the corresponding integral equation associated to (\ref{KdV2}) is
\begin{equation}
u(x,t) = e^{-t\partial_x^3}\phi(x)-\int_0^te^{-(t-\tau)\partial_x^3}w(x,\tau)d\tau,
\end{equation}
where $w$ is defined as in (\ref{defofw}).\\

Since we are only seeking for the local well-posedness, we may use a bump function to truncate time variable. 
Let $\psi$ be a bump function supported in $[-2, 2]$ with $\psi(t)=1, |t|\leq 1$, and let $\psi_\delta$ be 
$$
\psi_\delta(t) =\psi(t/\delta)\,.
$$
Then it suffices to find a local solution of
$$u(x,t) = \psi_\delta(t)e^{-t\partial_x^3}\phi(x)-\psi_\delta(t)\int_0^te^{-(t-\tau)\partial_x^3}w(x,\tau)d\tau.$$
Let $T$ be an operator given by
\begin{equation}\label{defofT}
Tu(x,t):= \psi_\delta(t)e^{-t\partial_x^3}\phi(x)-\psi_\delta(t)\int_0^te^{-(t-\tau)\partial_x^3}w(x,\tau)d\tau.
\end{equation}
We denote the first term (the linear term) in (\ref{defofT}) by ${\mathcal L}u$ and the second term
(the nonlinear term) by ${\mathcal N}u$. Henceforth we represent $Tu$ as $ {\mathcal L}u +{\mathcal N}u $.    

\begin{lemma}\label{estLu}
The linear term $ \mathcal L$ satisfies 
\begin{equation}\label{estLu1}
\|\mathcal L u\|_{Y_s}\leq C\|\phi\|_{H^s}\,. 
\end{equation}
Here $C$ is a constant independent of $\delta$. 
\end{lemma}

\begin{proof}
Notice that
$$\widehat{\mathcal L u}(n,\lambda) = \widehat{\phi}(n){\mathcal F}_{\mathbb R}{\psi_\delta}(\lambda-n^3) =
 \widehat{\phi}(n) \delta \mathcal F_{\mathbb R}{\psi}\left( \delta( \lambda-n^3)\right)  ,$$
Thus from the definition of $Y_s$ norm, 
\begin{align*}
\|\mathcal Lu\|_{Y_s} = &\left(\sum_n \int \langle n\rangle^{2s}\langle \lambda-n^3\rangle \left|\widehat{\phi}(n)
  \delta \mathcal F_{\mathbb R}{\psi}\left(\delta(\lambda-n^3)\right)\right|^2 d\lambda\right)^\frac{1}{2}\\
&+ \left(\sum_n \langle n\rangle^{2s}\left(\int \left|\widehat{\phi}(n) \delta \mathcal F_{\mathbb R}\psi\left(\delta(\lambda-n^3)\right) \right| d\lambda \right)^2\right)^\frac{1}{2}\,.
\end{align*}
Since $\psi$ is a Schwartz function, its Fourier transform is also a Schwartz function.  Using the fast decay property 
for the Schwartz function, we have 
$$
\|\mathcal L u\|_{Y_s}\leq C \left(\sum_n \langle n \rangle^{2s}\left|\widehat{\phi}(n)\right|^2\right)^\frac{1}{2} =C \|\phi\|_{H^s}.
$$

\end{proof}

\begin{lemma}\label{estNu}
The nonlinear term $ \mathcal N$ satisfies 
\begin{equation}\label{estNu1}
\|\mathcal N u\|_{Y_s}\leq C \left( \|w\|_{s, -\frac12} + \left(\sum_n\langle n\rangle^{2s}\left(\int\frac{|\widehat{w}(n,\lambda)|}{\langle\lambda-n^3\rangle}d\lambda\right)^2\right)^\frac{1}{2}  \right)\,,
\end{equation}
where $C$ is a constant independent of $\delta$. 
\end{lemma}

\begin{proof}
Represent $w$ as its space-time inverse Fourier transform so that 
we write 
\begin{equation}
\mathcal Nu(x,t) = -\psi_\delta(t)\int_0^t e^{-(t-\tau)\partial_x^3}\left(\sum_n\int\widehat{w}(n,\lambda)e^{inx}e^{i\lambda\tau}d\lambda\right)d\tau\,,
\end{equation} 
which is equal to
\begin{align*}
 & -\psi_\delta(t)\sum_n\int \widehat{w}(n,\lambda)\int_0^t e^{-(t-\tau)( in)^3}e^{inx}e^{i\lambda\tau}d\tau d\lambda\\
=&-\psi_\delta(t) \sum_n\int \widehat{w}(n,\lambda) e^{ inx}e^{in^3t}\ \frac{e^{i(\lambda-n^3)t}-1}{i(\lambda-n^3)}\ d\lambda\,.
\end{align*}
We decompose the nonlinear term $\mathcal Nu$ into three parts, denoted by $\mathcal N_1, \mathcal N_2, 
\mathcal N_3$ respectively. 

\begin{align*}
\mathcal Nu(x,t)
=& -\psi_\delta (t)\sum_n\int_{|\lambda-n^3|\leq\frac{1}{100\delta}}\widehat{w}(n,\lambda)e^{inx}e^{in^3t}\sum_{k\geq1}\frac{(it)^k}{k!}(\lambda-n^3)^{k-1}d\lambda\\
& +i \psi_\delta(t)\sum_n\int_{|\lambda-n^3| > \frac{1}{100\delta}}\frac{\widehat{w}(n,\lambda)}{\lambda-n^3}e^{inx}e^{i\lambda t} d\lambda\\
&- i\psi_\delta(t)\sum_n\left(\int_{|\lambda-n^3| >\frac{1}{100\delta}}\frac{\widehat{w}(n,\lambda)}{\lambda-n^3}d\lambda\right)e^{inx}e^{in^3t}\\
:=& \mathcal N_1u + \mathcal N_2u + \mathcal N_3u.
\end{align*}

First we estimate $\mathcal N_2$.  Using Fourier series expansion for $\psi$, we get
$$
\psi_\delta(t) =\sum_{m\in \mathbb Z} C_m e^{im t/\delta}\,.
$$
Here the coefficients $C_m$'s satisfy 
$$
 C_m \leq C (1+|m|)^{-100}\,. 
$$
Hence $\mathcal N_2u$ can be represent as
\begin{equation}
\mathcal N_2 u =i\sum_{m}C_m \sum_{n} e^{inx}\int_{|\lambda-n^3| > \frac{1}{100\delta}}\frac{\widehat{w}(n,\lambda)}{\lambda-n^3} e^{i(\lambda + m/\delta ) t} d\lambda
\end{equation}
By a change of variables $(\lambda+m/\delta)\mapsto \lambda $, 
\begin{equation}
\mathcal N_2 u =i\sum_{m}C_m \sum_{n} e^{inx}\int_{|\lambda-\frac{m}{\delta}-n^3| > \frac{1}{100\delta}}\frac{\widehat{w}(n,\lambda-m/\delta)}{\lambda-\frac{m}{\delta}-n^3} e^{i\lambda t} d\lambda
\end{equation}

Thus we estimate
\begin{equation}
\|\mathcal N_2 u\|_{s, \frac12}^2\leq  C\sum_m(1+|m|)^{-50}\sum_n \langle n\rangle^{2s}\int_{
 |\lambda-\frac{m}{\delta}-n^3|>\frac{1}{100\delta}} \frac{\langle \lambda-n^3\rangle \left| \wh w(n, \lambda-m/\delta) \right|^2}{|\lambda-\frac{m}{\delta}-n^3|^2} d\lambda\,.
\end{equation}
Changing variables again, we obtain
\begin{equation}
\|\mathcal N_2 u\|_{s, \frac12}^2\leq  C\sum_m(1+|m|)^{-50}\sum_n \langle n\rangle^{2s}\int_{
 |\lambda-n^3|>\frac{1}{100\delta}} \frac{\langle \lambda+\frac{m}{\delta}-n^3\rangle \left| \wh w(n, \lambda) \right|^2}{ \langle 
 \lambda-n^3\rangle^2} d\lambda\,.
\end{equation}
Notice that $ |\lambda-n^3|> \frac{1}{100\delta} $ implies 
\begin{equation}
 \langle \lambda +\frac{m}{\delta} -n^3\rangle \leq  200m\langle \lambda -n^3\rangle\,.
\end{equation}
We obtain immediately 
\begin{equation}\label{estN21}
\|\mathcal N_2 u\|_{s, \frac12}\leq C\| w\|_{s, -\frac{1}{2}}\,.
\end{equation}
On the other hand, 
$$
\sum_n \langle n\rangle^{2s}\!\left( \!\int  {| \wh{\mathcal N_2u}(n, \lambda)|}  d\lambda\right)^2\!\!
\!\!\leq C\sum_{m}\langle m\rangle^{-5}\!\sum_n\langle n \rangle^{2s}\!\!\left(\! \int_{|\lambda-\frac{m}{\delta}-n^3|>\frac{1}{100\delta}} \!\!  \frac{|\wh w(n, \lambda-m/\delta) |d\lambda}{  |\lambda-\frac{m}{\delta}-n^3|}  \right)^2,
$$
which is clearly bounded by
\begin{equation}\label{estN22}
\sum_n\langle n \rangle^{2s}\left(\int \frac{|\wh w(n, \lambda) |d\lambda}{  \langle\lambda-n^3\rangle}  \right)^2.
\end{equation}
Putting (\ref{estN21}) and (\ref{estN22}) together, we have 
\begin{equation}\label{estN2}
\| \mathcal N_2 u \|_{Y_s}\leq C\left ( \|w\|_{s,-\frac{1}{2}}+\left(\sum_n\langle n\rangle^{2s}\left(\int\frac{|\widehat{w}(n,\lambda)|}{\langle\lambda-n^3\rangle}d\lambda\right)^2\right)^\frac{1}{2} \right)\,.
\end{equation}
Let $A_n$ be defined by
\begin{equation}
A_n = \int_{|\lambda-n^3|\leq \frac{1}{100\delta}}\wh w(n, \lambda)(\lambda-n^3)^{k-1}d\lambda\,.
\end{equation}
Then $\mathcal N_1u$ can be written as
\begin{equation}
\mathcal N_1u(x, t) =-\sum_{k\geq 1}\frac{i^k}{k!}t^k\psi_\delta(t)\sum_n  A_n e^{inx}e^{in^3t}\,.
\end{equation}
Hence the space-time Fourier transform of $\mathcal N_1u$ satisfies 
\begin{equation}\label{N1Fest}
\left|\wh {\mathcal N_1u}(n, \lambda)\right|\leq \sum_{k\geq 1}\frac{1}{k!} |A_n|\left| \mathcal F_{\mathbb R}(\tilde{\psi_\delta})(\lambda-n^3)\right|\,,
\end{equation}
where $\tilde{\psi_\delta}(t)=t^k\psi_\delta(t)$.  Using the definition of Fourier transform, we have
$$
\left| \mathcal F_{\mathbb R}(\tilde{\psi_\delta})(\lambda-n^3)\right|\leq C\delta^{k+1}k^{3}\langle  \delta (\lambda-n^3)\rangle^{-3}\,.
$$
Thus 
\begin{eqnarray*}
\|\mathcal N_1u\|_{Y_s}^2 & \leq & 
 \sum_{k\geq 1} \frac{C}{k^5}\sum_n \langle n\rangle^{2s}|A_n|^2\delta^{2k}\int  \delta^2\langle \lambda-n^3\rangle
 \langle  \delta (\lambda-n^3)\rangle^{-6}
 d\lambda  \\
 &   &+  \sum_{k\geq 1}\frac{C}{k^5}\sum_n \langle n\rangle^{2s}|A_n|^2\delta^{2k}\left(\int  \delta \langle  \delta (\lambda-n^3)\rangle^{-3}  d\lambda \right)^2\\
&\leq &\sum_{k\geq 1} \frac{C}{k^5}\sum_n \langle n\rangle^{2s}|A_n|^2\delta^{2k}  \,.
\end{eqnarray*}
Clearly $A_n$ is bounded by
\begin{equation}
|A_n|\leq C\delta^{-k}\int \frac{|\wh w(n, \lambda)|}{\langle\lambda-n^3\rangle}d\lambda\,.
\end{equation}
Henceforth, we obtain
\begin{equation}\label{N1est}
\|\mathcal N_1u\|_{Y_s}\leq C\left(\sum_n\langle n\rangle^{2s}\left(\int\frac{|\widehat{w}(n,\lambda)|}{\langle\lambda-n^3\rangle}d\lambda\right)^2\right)^\frac{1}{2}.
\end{equation}

Similarly,  we may obtain
\begin{equation}\label{N3est}
\|\mathcal N_3u\|_{Y_s}\leq C\left(\sum_n\langle n\rangle^{2s}\left(\int\frac{|\widehat{w}(n,\lambda)|}{\langle\lambda-n^3\rangle}d\lambda\right)^2\right)^\frac{1}{2}.
\end{equation}

Therefore we complete the proof.

\end{proof}

\begin{proposition}\label{propTu}
Let $s>1/2$ and $T $ be the operator defined as in (\ref{defofT}).  Then there exits a
positive number $\theta$ such that
\begin{equation}\label{estofTu1}
\|Tu\|_{Y_s}\leq C\left(\|\phi\|_{H^s} + \delta^\theta \|u\|_{Y_s}^{k+1}\right)\,. 
\end{equation}
Here $C$ is a constant independent of $\delta$. 
\end{proposition}

\begin{proof}
Since $Tu=\mathcal L u +\mathcal N u$, Proposition \ref{propTu} follows from
Lemma {\ref{estLu}}, Lemma {\ref{estNu}} and Proposition {\ref{propofw}}.   
\end{proof}

Proposition \ref{propTu} yields that for $\delta$ sufficiently small, $T$ maps a ball in $Y_s$ 
into itself. Moreover, we write 
\begin{eqnarray*}
& & \left(u^k-\int_\mathbb{T}u^k dx \right)u_x- \left(v^k-\int_\mathbb{T}v^kdx\right)v_x\\
& =&
\left(u^k-\int_\mathbb{T}u^k dx \right)(u-v)_x + \left((u^k-v^k)-\int_\mathbb{T}(u^k-v^k) dx \right)v_x\\
\end{eqnarray*}
which equals to 
\begin{equation}\label{repu-v}
\left(u^k-\int_\mathbb{T}u^k dx \right)(u-v)_x + \sum_{j=0}^{k-1}
\left((u-v)u^{k-1-j}v^j-\int_\mathbb{T}(u-v)u^{k-1-j}v^j dx \right)v_x\,.
\end{equation}
For $k+1$ terms in (\ref{repu-v}), repeating similar argument as in the proof of Proposition {\ref{propofw}},  
one obtains, for $s>1/2$, 
\begin{equation}\label{contra}
\|Tu-Tv\|_{Y_s}\leq C\delta^\theta \left( \|u\|^k_{Y_s} +\sum_{j=1}^{k-1}\|u\|_{Y_s}^{k-1-j}\|v\|_{Y_s}^{j+1}\right) \|u-v\|_{Y_s}\,.
\end{equation}
Henceforth, for $\delta>0$ small enough, $T$ is a contraction and the local well-posedness 
follows from Picard's fixed-point theorem.  \\

\section{Proof of Proposition \ref{propofw}}\label{proofP1}
\setcounter{equation}0

From the definition of $w$ in (\ref{defofw}), we may write $\wh w(n, \lambda)$ as
\begin{equation}\label{w}
\sum_{\substack{m+n_1+\cdots+n_k=n\\n_1+\cdots+n_k\neq0}}m\int\widehat{u}(m,\lambda-\lambda_1-\cdots-\lambda_k)
\widehat{u}(n_1,\lambda_1)\cdots\widehat{u}(n_k,\lambda_k)d\lambda_1\cdots d\lambda_k.
\end{equation}

By duality, there exists a sequence 
 $\{A_{n, \lambda}\}$ satisfying 
\begin{equation}\label{AnL}
\sum_{n\in\mathbb Z} \int_{\mathbb R}|A_{n, \lambda}|^2 d\lambda \leq 1\,,
\end{equation}
and $\|w\|_{s, -\frac12}$ is bounded by
\begin{equation}\label{ws1}
\sum_{\substack{m+n_1+\cdots+n_k=n\\n_1+\cdots+n_k\neq0}}\!\int\frac{\langle n\rangle^s |m|}{\langle \lambda-n^3 \rangle^\frac{1}{2}}|\widehat{u}(m,\lambda-\lambda_1-\cdots-\lambda_k)|\cdot
 |\widehat{u}(n_1,\lambda_1)|\cdots|\widehat{u}(n_k,\lambda_k)||A_{n,\lambda}|
d\lambda_1\cdots d\lambda_k d\lambda.
\end{equation}

Since the $X_{s, b}$ is a restriction norm, we may assume that $u$ is supported in $\mathbb T\times[0, \delta]$. 
However, the inverse space-time Fourier transform $|\wh u|^\vee$ in general may not be a function 
with compact support. The following standard trick allows us to assume $|\wh u|^\vee$ has a compact support too.
In fact,   let $\eta$ be a bump function supported on $[-2\delta, 2\delta]$ and with $\eta(t)=1$ in $|t|\leq\delta$. Also $\widehat{\eta}$ is positive. Then $u=u\eta$ and 
$\wh u = \widehat{u}*\widehat{\eta}$.
Thus $|\wh u|\leq |\widehat{u}|*\widehat{\eta}=\left(|\widehat{u}|^\vee\eta\right)^\wedge$.  Whenever we need to
make $|\wh u|^\vee$ to be supported in a small time interval, we replace $ |\wh u|$ by $\left(|\widehat{u}|^\vee\eta\right)^\wedge$ since $ |\widehat{u}|^\vee\eta$ clearly is supported on $\mathbb T\times [-2\delta, 2\delta]$.
This will help us gain a positive power of $\delta$ in our estimates. 
Moreover, without loss of generality we can assume $|n_1|\geq|n_2|\geq\cdots\geq|n_k|$.\\

The trouble occurs mainly because of the factor $ |m|$ resulted from $\p_x u$. The idea is that either 
the factor $\langle \lambda-n^3 \rangle^{-\frac{1}{2}}$ can be used to cancel $|m|$, or $|m|$ can be 
distributed to some of $ \wh u$'s. More precisely, we consider three cases. 
\begin{eqnarray}
  &   |m| < 1000k^2|n_2|\,; &  \label{case1}\\
  &  1000k^2|n_2|\leq  |m| \leq  100k|n_1| \,; &\label{case2}\\
 &    |m|> 100k|n_1|\,. & \label{case3}
\end{eqnarray}

\subsection{Case (\ref{case1})}

This is the simplest case.   In fact, 
In this case,  it is easy to see that
\begin{equation}\label{shift1}
\langle n\rangle^s |m|  \leq   C \langle n_1\rangle^s \langle n_2\rangle^\frac{1}{2} \langle m\rangle^\frac{1}{2}.
\end{equation}
Let 
\begin{eqnarray}
& &F_1(x,t) = \sum_n \int \frac{|A_{n, \lambda}|}{\langle \lambda-n^3\rangle^{\frac12}} e^{i\lambda t}  e^{inx} d\lambda\,;
  \label{defofF} \\
& &
G(x,t) = \sum_n \int \langle n\rangle^\frac{1}{2}|\widehat{u}(n,\lambda)|    e^{i\lambda t}  e^{inx} d\lambda\,
\label{defofG}\\
& &
H(x,t) = \sum_n \int \langle n\rangle^s  |\wh u(n,\lambda)|    e^{i\lambda t}  e^{inx} d\lambda\, \label{defofH}\\
& &
U(x,t) = \sum_n \int    |\wh u(n,\lambda)|    e^{i\lambda t}  e^{inx} d\lambda\,
\label{defofU}
\end{eqnarray}

Then using (\ref{shift1}),  we can estimate (\ref{ws1}) by
\begin{equation}\label{estWs1}
C\!\!\!\!\!\sum_{m+n_1+\cdots+n_k=n}\!\int\!\widehat{F_1}(n,\lambda)\widehat{G}(m,\lambda-\lambda_1-\cdots-\lambda_k)\widehat{H}(n_1,\lambda_1)\widehat{G}(n_2,\lambda_2)
\prod_{j=3}^k \wh U(n_j, \lambda_j) d\lambda_1\cdots d\lambda_k d\lambda\,,
\end{equation}
which clearly equals 
$$
C \int_{\mathbb{T}\times\mathbb{R}} F_1(x,t)G(x,t)^2H(x,t)U(x,t)^{k-2}dxdt \,.
$$
Apply H\"older inequality to majorize it by
$$
C\|F_1\|_4\|G\|_{6+}^2\|H\|_4\|U\|_{6(k-2)-}^{k-2}\,.
$$
Since $U$ is supported on $\mathbb T\times [-2\delta, 2\delta]$, one more use of H\"older inequality yields 
\begin{equation}
(\ref{ws1}) \leq C\delta^\theta\|F_1\|_4\|G\|_{6+}^2\|H\|_4\|U\|_{6(k-2)}^{k-2}\,.
\end{equation}

Let us recall some useful local embedding facts on $X_{s, b}$. 
\begin{eqnarray}
& &X_{0,\frac{1}{3}}\subseteq L_{x,t}^4 \,,\,\,\,\, X_{0+, \frac{1}{2}+}\subseteq L^{6}_{x,t} \,,\,  \,\,\,  \,(t\ \text{local} ) \label{emb1} \\
& &X_{\alpha,\frac{1}{2}}\subseteq L_{x,t}^q,\quad 0<\alpha<\frac{1}{2},\ 2\leq q<\frac{6}{1-2\alpha}\quad   (t\ \text{local}), \label{emb2}\\
& &X_{\frac{1}{2}-\alpha, \frac{1}{2}-\alpha}\subseteq L_t^qL_x^r,\quad 0<\alpha<\frac{1}{2},\ 2\leq q, r<\frac{1}{\alpha} \label{emb3}.
\end{eqnarray}
The two embedding results in (\ref{emb1}) are consequences of the discrete restriction estimates on $L^4$ and $L^6$ respectively. (\ref{emb2}) and (\ref{emb3}) follow by interpolation (see \cite{Tao} for details).
(\ref{emb1}) yields 
$$
\|F_1\|_4\leq C\|F_1\|_{0, \frac13} \leq C\left(  \sum_n \int |A_{n, \lambda}|^2 d\lambda\right)^{1/2}\leq C\,,
$$
and 
$$
\|H\|_4\leq C\|H\|_{0,\frac{1}{3}}\leq C\|u\|_{s,\frac{1}{2}}\leq C\|u\|_{Y_s}\,.
$$
(\ref{emb2}) implies 
$$
\|G\|_{6+}\leq C\|G\|_{0+,\frac{1}{2}}\leq C\|u\|_{s,\frac{1}{2}}\leq C\|u\|_{Y_s}\,.
$$
Using (\ref{emb3}), we get
$$
\|U\|_{6(k-2)}\leq C\|U\|_{\frac{1}{2}-,\frac{1}{2}-}\leq C\|u\|_{s,\frac{1}{2}}\leq C\|u\|_{Y_s}.
$$ 
Henceforth, we have, for the case (\ref{case1}),
\begin{equation}\label{case1est}
(\ref{ws1})\leq C\delta^\theta \|u\|_{Y_s}^{k+1}\,. 
\end{equation}

\subsection{Case (\ref{case2})}\label{Hicase2}

In this case, we should further consider two subcases.
\begin{eqnarray}
  & |m+n_1| \leq  1000k^2|n_2|&   \label{subcase21}\\
&|m+n_1| > 1000k^2|n_2|   &  \label{subcase22}
\end{eqnarray}   
In the subcase (\ref{subcase21}), we use the triangle inequality to get
\begin{equation}
 |n|=|m+n_1+n_2+\cdots +n_k| \leq C|n_2|
\end{equation}
Hence, we have
\begin{equation}
\langle n \rangle^s |m|\leq C\langle n_2\rangle^s \langle m\rangle^{\frac12}\langle n_1\rangle^\frac12\,. 
\end{equation}
Thus this subcase can be treated exactly the  same as the case (\ref{case1}). We omit the details. \\

For the subcase (\ref{subcase22}), the crucial arithmetic observation is 
\begin{equation}\label{arith}
n^3-(m^3+n_1^3+\cdots + n_k^3) =  3(m+n_1)(m+ a)(n_1+a) + a^3-(n_2^3+\cdots +n_k^3)\,,
\end{equation}
where $a=n_2+\cdots +n_k$.  This observation can be easily verified since $n=m+n_1+\cdots+n_k$. 
From (\ref{case2}) and (\ref{subcase22}), we get
\begin{equation}\label{diffL}
\left| n^3- (m^3+n_1^3+\cdots + n_k^3) \right|\geq Ck^2 \langle n_2\rangle |m||n_1| \geq Ck|m|^2\,.
\end{equation}
This implies at least one of following statements holds:
\begin{eqnarray}
 &  \left| \lambda- n^3\right|\geq C |m|^2 \,, &  \label{I} \\
&   \left |(\lambda-\lambda_1-\cdots -\lambda_k) -m^3\right|\geq C|m|^2\,, &  \label{II}\\
&  \exists  i \in \{1, \cdots, k\}\,\,\, {\text {such that}}\,\,\, |\lambda_i - n_i^3|\geq  C|m|^2\,.& \label{III}
\end{eqnarray}

For (\ref{I}), (\ref{ws1}) can be bounded by
\begin{equation}\label{I-est}
\sum_{m+n_1+\cdots+n_k=n}\int \langle n_1 \rangle^s |\widehat{u}(m,\lambda-\lambda_1-\cdots-\lambda_k)|
 |\widehat{u}(n_1,\lambda_1)|\cdots|\widehat{u}(n_k,\lambda_k)||A_{n,\lambda}|
d\lambda_1\cdots d\lambda_k d\lambda.
\end{equation}
Let $F_2$ be defined by
\begin{equation}\label{defofF1}
F_2(x,t) = \sum_n \int |A_{n, \lambda}|e^{i\lambda t}  e^{inx} d\lambda\,.
\end{equation}
Then we represent (\ref{I-est}) as
\begin{equation}\label{I-est2}
\sum_{m+n_1+\cdots+n_k=n}\int  \wh F_2(n, \lambda) \wh U(m, \lambda-\lambda_1-\cdots-\lambda_k)
\wh H(n_1, \lambda_1)\prod_{j=2}^k\wh U(n_j, \lambda_j) d\lambda_1\cdots d\lambda_k d\lambda\,.
\end{equation}
Here $H$ and $U$ are functions defined in (\ref{defofH}) and (\ref{defofU}) respectively.
Clearly (\ref{I-est2}) equals 
\begin{equation}\label{I-est3}
\int_{\mathbb T\times \mathbb R} F_2(x, t) H(x,t) U(x, t)^{k} dx dt\,.
\end{equation}
Utilizing H\"older inequality, we estimate it further by
\begin{equation}
 \|F_2\|_2 \|H\|_4\|U\|_{4k}^k \leq C\delta^\theta \|u\|_{Y_s}^{k+1}\,.
\end{equation}
This yields the desired estimate for the subcase (\ref{I}).\\

For the subcase of (\ref{II}), (\ref{ws1}) is estimated by
\begin{eqnarray*}
&  &\sum_{m+n_1+\cdots+n_k=n}\int\frac{\langle n_1\rangle^s |A_{n, \lambda}|}{\langle \lambda-n^3 \rangle^\frac{1}{2}}\ \langle (\lambda-\lambda_1-\cdots-\lambda_k)-m^3\rangle^{\frac12} |\widehat{u}(m,\lambda-\lambda_1-\cdots-\lambda_k)|
 \\
&  &\,\,\,\,\,\,\,\,\,\hspace{2cm} \cdot |\widehat{u}(n_1,\lambda_1)|\cdots|\widehat{u}(n_k,\lambda_k)|
d\lambda_1\cdots d\lambda_k d\lambda\,, 
\end{eqnarray*}
which is equal to
\begin{equation}\label{II-est}
 \int_{\mathbb T\times \mathbb R} F_1(x,t) G(x,t) H(x,t) U^{k-1}(x,t) dx dt\,.
\end{equation}
Apply H\"older inequality to control (\ref{II-est}) by
\begin{equation}\label{II-est1}
\|F_1\|_4\|G\|_4\|H\|_4 \|U\|_{4(k-1)}^{k-1} \leq C\delta^\theta \|u\|_{Y_s}^{k+1}\,.
\end{equation}
This completes the estimate for the subcase (\ref{II}).\\

For the contribution of (\ref{III}), we only consider $|\lambda_2-n_2^3|\geq C|m|^2$ without 
loss of generality for $i\in\{2, \cdots, k\}$.  
This is because the $|\lambda_1-n_1^3|\geq C|m|^2$ case can be handled similarly as (\ref{II}). 
Hence, in this case, (\ref{ws1}) can be bounded by
$$
\sum_{m+n_1+\cdots+n_k=n}\int\frac{\langle n_1\rangle^s|A_{n,\lambda}|}{\langle \lambda-n^3 \rangle^\frac{1}{2}}\ \langle \lambda_2-n_2^3\rangle^\frac{1}{2} |\widehat{u}(m,\lambda-\lambda_1-\cdots-\lambda_k)|
\prod_{j=1}^k|\widehat{u}(n_j,\lambda_j)|
d\lambda_1\cdots d\lambda_k d\lambda.
$$
Now set a function $I$ by
\begin{equation}\label{defofI}
I(x,t) = \sum_n\int \langle \lambda-n^3\rangle^\frac{1}{2}|\wh u(n, \lambda)|e^{i\lambda t} e^{inx} d\lambda\,.
\end{equation}
Then we estimate (\ref{ws1}) by
\begin{equation}\label{est-III1}
\int_{\mathbb T\times \mathbb R} F_1(x,t)H(x,t)I(x,t)U^{k-1}(x,t)dxdt\,,
\end{equation}
which is majorized by
\begin{equation}
 \|F_1\|_4\|H\|_4\|I\|_2\|U\|_\infty^{k-1}\,.
\end{equation}

Notice this time we cannot simply use H\"older's inequality to get $\delta$ as we did before because there is no way of making any above 4 or 2 even a little bit smaller. But this can be fixed as follows.

First observe that  
$$
\|u\|_{0, 0}\leq \delta^{1/2}\|u\|_{L^2_xL^\infty_t}\leq C\delta^{1/2}\|u\|_{0, \frac12+}\,,
$$
for $u$ is supported in a $\delta$-sized interval in time variable. Thus by interpolation,
we get
\begin{equation}\label{gain}
\|u\|_{0, \frac13}\leq C\delta^{\frac16-}\|u\|_{0, \frac12}\,. 
\end{equation}
Since $U$ can be assumed to be a function supported in a $\delta$-sized time interval, 
we may put the same assumption to $H$. 
Henceforth, we have
\begin{equation}\label{III-est2}
\|H\|_4\leq C\|H\|_{0, \frac13}\leq C\delta^{\frac16-}\|H\|_{0, \frac12}\leq C\delta^{\frac16-}\|u\|_{Y_s}\,. 
\end{equation}
Also note that
\begin{equation}\label{III-est3}
\|I\|_2\leq \|u\|_{0,\frac12} \leq \|u\|_{Y_s}\,.
\end{equation}
and
\begin{equation}\label{III-est4}
\|U\|_{\infty}\leq C\|u\|_{Y_s}\,. 
\end{equation}
From (\ref{III-est2}), (\ref{III-est3}) and (\ref{III-est4}), we can estimate (\ref{ws1}) by
$ C\delta^{\frac16-}\|u\|_{Y_s}^{k+1} $ as desired. Therefore we finish our discussion for 
the case (\ref{case2}).

\subsection{Case (\ref{case3})} The arithmetic observation (\ref{arith}) again plays an 
important role.  In this case, let us further consider two subcases.
\begin{eqnarray}
  &  |m|^2 \leq 1000k^2|n_2|^2 |n_3|    &   \label{subcase31}\\
  &  |m|^2 > 1000k^2|n_2|^2 |n_3|   &  \label{subcase32}
\end{eqnarray}   

For the contribution of (\ref{subcase31}), we observe that from (\ref{subcase31}), 
$$
 |m|^2\leq C|n_1||n_2||n_3|\,,
$$
since $|n_2|\leq |n_1|$. Henceforth we have
\begin{equation}\label{shift3}
|m| = |m|^{\frac13}|m|^{\frac23}\leq C|m|^{\frac13}|n_1|^{\frac13}|n_2|^{\frac13}|n_3|^{\frac13}\,.
\end{equation}
This implies immediately 
\begin{equation}\label{shift31}
 \langle n\rangle^s |m|\leq C|m|^{s+1}\leq \langle m\rangle^\frac{s+1}{3} \langle n_1\rangle^\frac{s+1}{3} \langle n_2\rangle^\frac{s+1}{3} \langle n_3\rangle^\frac{s+1}{3}.
\end{equation}
Introduce a new function $H_1$ defined by
\begin{equation}\label{defofH1}
 H_1(x, t) =\sum_n \int_{\mathbb R} \langle n\rangle^{\frac{s+1}{3}}|\wh u(n, \lambda)| e^{i\lambda t} e^{inx} d\lambda\,.
\end{equation}
As before, in this case, we bound (\ref{ws1}) by
\begin{equation}\label{estCase31}
\int_{\mathbb T\times \mathbb R} F_1(x,t)H_1^4(x,t)U^{k-3}(x,t) dx dt\,.
\end{equation}
Then H\"older inequality yields 
\begin{equation}
(\ref{ws1})\leq C\delta^\theta\|F_1\|_4\|H_1\|_{6+}^4\|U\|_{12(k-3)}^{k-3}\,.
\end{equation}
$\|H_1\|_{6+}\leq C\|u\|_{Y_s}$ because $\frac{s+1}{3}< s$ for $s>1/2$. Hence we obtain 
the desired estimate for the subcase (\ref{subcase31}).\\

We now turn to the contribution of (\ref{subcase32}). Clearly we have
\begin{equation}\label{arith2}
|(n_2+\cdots+n_k)^3-(n_2^3+\cdots+n_k^3)| \leq 10k|n_2|^2|n_3|\,,
\end{equation}
since $|n_2|\geq |n_3|\geq\cdots\geq |n_k|$. 
From the crucial arithmetic observation (\ref{arith}), (\ref{arith2}), and (\ref{subcase32}), we have
\begin{equation}\label{arith3}
 \left|n^3 -\left(m^3+n_1^3+\cdots+n_k^3\right)\right| \geq  Ck|m|^2\,.
\end{equation} 
This is same as (\ref{diffL}). Hence again we reduce the problems to (\ref{I}), (\ref{II}), and (\ref{III}), which are all done in Subsection \ref{Hicase2}. Therefore we finish the case of (\ref{case3}).

Putting all cases together, we obation
\begin{equation}\label{ws1est}
\|w\|_{s, -\frac12}\leq C\delta^\theta \|u\|_{Y_s}^{k+1}\,.
\end{equation}

Finally we need to estimate
\begin{equation}\label{extra}
\left(\sum_n\langle n\rangle^{2s}\left(\int\frac{|\widehat{w}(n,\lambda)|}{\langle\lambda-n^3\rangle}d\lambda\right)^2\right)^\frac{1}{2}.
\end{equation}
Let $\{A_n\}$ be a sequence $\{A_n\}$ with $\left(\sum_n|A_n|^2\right)^\frac{1}{2}\leq1$.
By duality, it suffices to  estimate
\begin{equation}\label{extra1}
\sum_{\substack{m+n_1+\cdots+n_k=n\\n_1+\cdots+n_k\neq0}}\int\frac{\langle n\rangle^s |m|}{\langle \lambda-n^3 \rangle}|\widehat{u}(m,\lambda-\lambda_1-\cdots-\lambda_k)| |\widehat{u}(n_1,\lambda_1)|\cdots|\widehat{u}(n_k,\lambda_k)||A_n| d\lambda_1\cdots d\lambda_k d\lambda\,.
\end{equation}
Again, without loss of generality, we can assume $|n_1|\geq\cdots\geq|n_k|$. 
We still go through the cases used previously. Almost all cases are similar and there are only two exceptions. In fact, we only need to replace $F_1$ by $F_3$ in each case where $\|F_1\|_4$ is employed. Here $F_3 $ is given by
\begin{equation}\label{defofF2}
\sum_n\int_{\mathbb R} \frac{|A_n|}{\langle\lambda-n^3\rangle} e^{i\lambda t}e^{inx} d\lambda\,.
\end{equation}
Then all those cases can be done because
\begin{equation}\label{estF2}
\|F_3\|_4\leq C\|F_3\|_{0,\frac{1}{3}}=\left(\sum_n|A_n|^2\int\frac{1}{\langle\lambda-n^3\rangle^\frac{4}{3}}d\lambda\right)^\frac{1}{2} \leq C\,.
\end{equation}
The only exceptions are 
\begin{eqnarray}
& |\lambda - n^3 |\geq C|n_1||m| \,\,\,{\text{and }}\,\,\, |n_2|\ll|m|\leq C|n_1| & \label{ex1}\\
& |\lambda - n^3 |\geq Cm^2 \,\,\,{\text{and }}\,\,\,  |m|\gg |n_1|   & \label{ex2}
\end{eqnarray}
For the case of (\ref{ex1}), we define
\begin{equation}\label{defofF3}
F_4(x,t)= \sum_n \int_{\mathbb R} \frac{\langle n\rangle^{\frac12} \Id_{\{|\lambda-n^3|\geq C\langle n\rangle\}}}
{|\lambda-n^3 |} e^{i\lambda t} e^{inx}d\lambda 
\end{equation}
A direct calculation gives 
\begin{equation}\label{estF3}
\|F_4\|_2\leq \left(\sum_n \int_{|\lambda-n^3|\geq C\langle n\rangle} \frac{\langle n\rangle|A_n|^2}{|\lambda-n^3|^2} d\lambda \right)^{1/2}\leq C\,.
\end{equation}
In this case, clearly
\begin{equation}\label{shift5}
\langle n\rangle^s|m|\leq \langle n\rangle^{\frac12}\langle n_1\rangle^{s}\langle m\rangle^{\frac12}\,. 
\end{equation}
Then (\ref{extra1}) is dominated by
\begin{equation}\label{est-ex1}
\int_{\mathbb T\times \mathbb R} F_4(x, t)G(x,t)H(x,t) U^{k-1}(x,t)dx dt\,.
\end{equation}
By a use of H\"older inequality and (\ref{estF3}), one gets
\begin{equation}\label{est-ex2}
(\ref{extra1})\leq C\|F_4\|_2\|H\|_4\|G\|_6\|U\|_{12(k-1)}^{k-1}\leq C\delta^\theta\|u\|_{Y_s}^{k+1}\,.
\end{equation}
This finishes the proof for the case (\ref{ex1}).\\

For the contribution of (\ref{ex2}), we set 
\begin{equation}\label{defofF4}
F_5(x,t)=\sum_n \int_{\mathbb R} \frac{\langle n\rangle  \Id_{\{|\lambda-n^3|\geq C\langle n\rangle^2\}}}
{|\lambda-n^3 |} e^{i\lambda t} e^{inx}d\lambda \,.
\end{equation}
Clearly 
\begin{equation}\label{estF4}
\|F_5\|_2\leq \left(\sum_n \int_{|\lambda-n^3|\geq C\langle n\rangle^2} \frac{\langle n\rangle^2|A_n|^2}{|\lambda-n^3|^2} d\lambda \right)^{1/2}\leq C\,.
\end{equation}
In this case, we have $|\lambda-n^3|\geq C\langle n\rangle^2 $ since $|n|\sim |m|$, henceforth, by the observation of 
$$
 \langle n\rangle^s |m|\leq C \langle m\rangle^s \langle n\rangle\,,
$$
we estimate (\ref{extra1}) by
\begin{equation}\label{est-ex3}
\int_{\mathbb T\times \mathbb R} F_5(x,t)H(x,t)U^k(x,t)dx dt\,.
\end{equation}
Using H\"older inequality and (\ref{estF4}), we have
\begin{equation}\label{est-ex21}
(\ref{extra1}) \leq C\|F_5\|_2\|H\|_4\|U\|_{4k}^{4k}\leq C\delta^{\theta}\|u\|_{Y_s}^{k+1}\,,
\end{equation}
as desired. Hence
\begin{equation}\label{extra-est}
\left(\sum_n\langle n\rangle^{2s}\left(\int\frac{|\widehat{w}(n,\lambda)|}{\langle\lambda-n^3\rangle}d\lambda\right)^2\right)^\frac{1}{2} \leq C\delta^{\theta}\|u\|_{Y_s}^{k+1}.
\end{equation}
Therefore we complete the proof of Proposition {\ref{propofw}} by combining (\ref{ws1est}) and (\ref{extra-est}). \\

\section{Proof of Theorem \ref{LWP-KdV-F}}\label{LWP2}
\setcounter{equation}0

The argument is similar to those in Section \ref{LWP1}. 
By using a gauge transform as in (\ref{Gauge1}) with $v^k$ replaced by $F(v)$, the well-posedness of
(\ref{KdV0}) is equivalent to the well-posedness of the following equation:
\begin{equation}\label{KdV0-1}
\begin{cases}
u_t+u_{xxx}+ \left(F(u)-\int_{\mathbb T} F(u) dx\right)u_x=0\\
u(x,0)=\phi(x),\qquad x\in\mathbb{T},\ t\in\mathbb{R}\,.
\end{cases}
\end{equation}

Now the nonlinear function $w$ is defined by
\begin{equation}\label{defofwF}
 w =\p_x u \left( F(u)-\int_{\mathbb T} F(u) dx \right)\,.
\end{equation}

Let $T_F$ be an operator given by
\begin{equation}\label{defofT-F}
T_Fu(x,t):= \psi_\delta(t)e^{-t\partial_x^3}\phi(x)-\psi_\delta(t)\int_0^te^{-(t-\tau)\partial_x^3}w(x,\tau)d\tau.
\end{equation}

As in Section \ref{LWP1}, the local well-posedness relies on the following 
proposition.

\begin{proposition}\label{propF}
Let $s>1/2$. There exists $\theta>0$ such that, for the nonlinear function $w$ 
 given by (\ref{defofwF}) and any $ u$ satisfying $\|u\|_{Y_s}\leq C_0\|\phi\|_{H^s}$,   
\begin{equation}\label{Main-estF}
\|w\|_{s, -\frac12} + \left(\sum_n\langle n\rangle^{2s}\left(\int\frac{|\widehat{w}(n,\lambda)|}{\langle\lambda-n^3\rangle}d\lambda\right)^2\right)^\frac{1}{2} \leq C(\|\phi\|_{H^s}, F)
\delta^\theta \|u\|_{Y_s}^{4},
\end{equation}
provided $F\in C^5$. 
Here $C_0$ is a suitably large constant, and $C(\|\phi\|_{H^s}, F)$ is a constant independent of $\delta$ and $u$, but 
may depend on $\|\phi\|_{H^s}$ and $F$. 
\end{proposition}

The constant $C(\|\Phi\|_{H^s}, F)$ will be specified in the proof of Proposition \ref{propF}. 
We postpone the proof of Proposition \ref{propF} to Section \ref{proofPropF}, and return to the proof of Theorem \ref{LWP-KdV-F}.
Proposition \ref{propF} implies that for $\delta$ sufficiently small, $T_F$ maps a ball $\{u\in Y_s: 
 \|u\|_{Y_s}\leq C_0\|\phi\|_{H^s}\}$ into itself. Moreover, using Lemma {\ref{estNu}} and repeating similar argument as in the proof of Proposition 
{\ref{propF}},  one obtains, for $s>1/2$ and $F\in C^5$, 
\begin{equation}\label{contra-F}
\|T_Fu-T_Fv\|_{Y_s}\leq  \delta^\theta  C(\|\phi\|_{H^s}, F) \|u-v\|_{Y_s}\,.
\end{equation}
for all $u, v$ in the ball $\{u\in Y_s: 
 \|u\|_{Y_s}\leq C_0\|\phi\|_{H^s}\}$.
Therefore, for $\delta>0$ small enough, $T_F$ is a contraction on the ball and the local well-posedness 
again follows from Picard's fixed-point theorem. This completes the proof of Theorem {\ref{LWP-KdV-F}}.  \\

\section{Proof of Proposition \ref{propF}}\label{proofPropF}
\setcounter{equation}0

First we introduce a decomposition of $F(u)$, which was used by Bourgain.  
Let $K$ be a dyadic number, and define a Fourier multiplier operator $P_K$ by
setting
\begin{equation}\label{defofPK}
P_Ku(x, t) = \int \psi_K(y)u(x-y, t)dy\,.
\end{equation}
Here the Fourier transform of $\psi_K$ is a standard bump function supported on $[-2K, 2K]$ and $\wh {\psi_K} (x)=1$ for $x\in[-K, K]$.
Let $u_K$ denote the Littlewood-Paley Fourier multiplier, that is, 
\begin{equation}\label{defofuK}
u_K = P_Ku- P_{K/2}u\,.
\end{equation}
Then we may decompose $F(u)$ by
\begin{eqnarray*}
F(u) & = &\sum_{K} \left( F(P_Ku)-F(P_{K/2}u)\right)\\
    & =& \sum_K F_1(P_Ku, P_{K/2}u) u_K + R_1\,,
\end{eqnarray*}
where $R_1$ is a function independent of the space variable $x$. Repeating this procedure for
$F_1$, we obtain 
\begin{eqnarray*}
 F(u) &= & \sum_{K_1\geq K_2}F_2(P_{2K_2}u, \cdots, P_{K_2/4}u)u_{K_1}u_{K_2} + \sum_{K_1}R_{2}u_{K_1} + R_1   \\
 & = & \sum_{K_1\geq K_2\geq K_3}F_3(P_{4K_3}u, \cdots, P_{K_3/8}u)u_{K_1}u_{K_2}u_{K_3}  \\
 & &   + \sum_{K_1\geq K_2}R_{3}u_{K_1}u_{K_2}+ \sum_{K_1}R_{2}u_{K_1} + R_1  
\end{eqnarray*}
where $R_1, R_{2}, R_{3}$ are functions independent of the space variable. \\
Set
\begin{equation}\label{G3}
 G_{K_3}(x,t) = F_3(P_{4K_3}u, \cdots, P_{K_3/8}u)\,.
\end{equation}

Hence we represent $w$ defined in (\ref{defofwF}) as 
\begin{eqnarray*}
w & =& \sum_{K_0, K_1\geq K_2\geq K_3}\p_xu_{K_0}
     \left( u_{K_1}u_{K_2}u_{K_3} G_{K_3} -\int_{\mathbb T}u_{K_1}u_{K_2}u_{K_3} G_{K_3} dx\right) \\
 & & + \sum_{K_0, K_1\geq K_2}\p_xu_{K_0} \left(u_{K_1}u_{K_2}-\int_{\mathbb T} u_{K_1}u_{K_2}dx  \right)R_{3}\\
 & & + \sum_{K_0, K_1}\p_xu_{K_0} \left(u_{K_1}-\int_{\mathbb T} u_{K_1}dx  \right)R_{2}\,.
\end{eqnarray*}
The main contribution of $w$ is from the first term. The remaining terms can be handled by the method 
presented in Section \ref{proofP1} because $R_{2}, R_{3}$ are functions independent of the space variable $x$ (actually they only depend on the conserved quantity $\int_\mathbb T udx$). Hence in what follows we will only focus on estimating the first term--the most difficult one.  
Denote the first term by $w_1$, i.e., 
\begin{equation}\label{defofw1}
w_1= \sum_{K_0, K_1\geq K_2\geq K_3}\p_xu_{K_0}
     \left( u_{K_1}u_{K_2}u_{K_3} G_{K_3} -\int_{\mathbb T} u_{K_1}u_{K_2}u_{K_3} G_{K_3} dx\right)\,.
\end{equation}

We should prove
\begin{equation}
\|w_1\|_{s, -\frac12} + 
  \left( \sum_{n} \langle n\rangle^{2s} \left( \int \frac{|\wh w_1(n, \lambda)|}{\langle \lambda-n^3\rangle}d\lambda \right)^2\right)^{1/2}\leq  \delta^{\theta} C(\|\phi\|_{H^s}, F) \|u\|_{Y_s}^{4}\,.
\end{equation}
In order to specify the constant $ C(\|\phi\|_{H^s}, F)$,  we define ${\mathfrak M} $ by setting
\begin{equation}
\mathfrak M =\sup \left\{ |D^\alpha F_3(u_1,\cdots, u_6)|  :  u_j\,\,\,{\text{satisfies}} \,\,\, \|u_j\|_{Y_s}\leq C_0\|\phi\|_{H^s}\,\,\, {\text{for all} }\,\,\, j=1, \cdots, 6;\,\,\, \alpha \right\}\,.
\end{equation}
Here $ D^\alpha=\p_{x_1}^{\alpha_1}\cdots\p_{x_6}^{\alpha_6} $ and $\alpha $ is taken over all tuples $ (\alpha_1, \cdots, \alpha_6)\in (\mathbb N \cup \{0\})^6$ with 
$0\leq \alpha_j\leq 2$ for all $j\in\{1, \cdots, 6\}$.
$\mathfrak M$ is a real number. This is because,  for $s>1/2$, $\|u\|_{Y_s}\leq 2\|\phi\|_{H^s}$ yields that $u$ is bounded by $C\|\phi\|_{H^s}$,  and the previous claim follows from  $F_3\in C^2$. \\

In order to bound $\|w_1\|_{s, -\frac12}$, by duality, it suffices to bound 
\begin{equation}\label{w1s1}
\begin{aligned}
 & \sum_{\substack { K_0, K_1\geq K_2\geq K_3 \\  n_0+n_1+n_2+n_3+m=n\\
n_1+n_2+n_3+m\neq 0 }}\int \frac{A_{n, \lambda}\langle n\rangle^s n_0}{\langle \lambda -n^3\rangle^{\frac12}}
 {\wh u_{K_0}}(n_0, \lambda-\lambda_1-\lambda_2-\lambda_3-\mu) \\
 & \hspace{2cm}\cdot \prod_{j=1}^3{\wh u_{K_j}}(n_j, \lambda_j) {\wh{G_{K_3}}}(m, \mu)d\lambda_1\cdots d\lambda_4 d\lambda
 d\mu\,,
\end{aligned}
\end{equation}
where $A_{n, \lambda}$ satisfies 
$$
 \sum_n \int |A_{n, \lambda}|^2 d\lambda =1\,. 
$$


The trouble maker is $G_{K_3}$ since there is no way to find a suitable upper bound for
its $X_{s, b}$ norm. Because of this, the method in Section \ref{proofP1} is no more valid, and
 we have to treat $m$ and $\mu$ differently from $n$ and $\lambda$ respectively. 
A delicate analysis must be done for overcoming the difficulty caused by $G_{K_3}$. For simplicity, we assume that $\delta=1$. One can 
modify the argument to gain a decay of $\delta^\theta$ by using the technical treatment from
Section \ref{proofP1}.\\

For a dyadic number $M$, define the Littlewood-Paley Fourier multiplier by
\begin{equation}\label{defofgKM}
g_{K_3, M} = P_MG_{K_3}- P_{M/2}G_{K_3}= (G_{K_3})_M\,.
\end{equation}
Let $v$ be defined by
\begin{equation}\label{defofv}
v(x,t)=\sum_n\int \frac{A_{n, \lambda}}{\langle \lambda-n^3\rangle^{\frac12}} e^{i\lambda t} e^{inx}d\lambda\,.
\end{equation}
To estimate (\ref{w1s1}), it suffices to estimate
\begin{equation}\label{w1s2}
\begin{aligned}
 & \sum_{\substack { K, K_0, K_1\geq K_2\geq K_3, M \\  n_0+n_1+n_2+n_3+m=n\\
n_1+n_2+n_3+m\neq 0 }}\int \wh{\p_x^s v_K}(n, \lambda)\wh{\p_x u_{K_0}}(n_0, \lambda-\lambda_1-\lambda_2-\lambda_3-\mu) \\
&  \hspace{2cm}\prod_{j=1}^3{\wh u_{K_j}}(n_j, \lambda_j) {\wh{g_{K_3, M}}}(m, \mu)d\lambda_1\cdots d\lambda_4 d\lambda d\mu\,.
\end{aligned}
\end{equation}
Here $K$ is a dyadic number.

As we did in Section \ref{proofP1}, we consider three cases:
\begin{eqnarray}
  &   K_0 < 2^{100} K_2\,; &  \label{case1F}\\
  &  2^{100} K_2\leq  K_0 \leq  2^{10} K_1 \,; &\label{case2F}\\
 &    K_0 > 2^{10}K_1 \,. & \label{case3F}
\end{eqnarray}
The rest part of the paper is devoted to a proof of these three cases. In what follows, we will only provide the details for the estimates of $\|w_1\|_{s, -\frac12}$ with $1/2<s<1$ (the case $s\geq 1$ is
easier).  For the desired estimate of
 $$\left( \sum_{n} \langle n\rangle^{2s}\left( \int \frac{|\wh w_1(n, \lambda)|}{\langle \lambda-n^3\rangle}d\lambda \right)^2\right)^{1/2}\,,$$ simply 
replace $v$ by 
\begin{equation}\label{defofv1}
 v_1(x, t)=\sum_n\int \frac{C_{n, \lambda} A_n }{\langle \lambda -n^3\rangle} e^{i\lambda t} e^{inx} d\lambda\,,
\end{equation} 
and then the desired estimate follows similarly.  Here $C_{n, \lambda}\in \mathbb C$ satisfies 
$\sup_{\lambda}|C_{n, \lambda}| \leq 1$ and $\{A_n\}$ satisfies
$\sum_{n } |A_n|^2\leq 1$.  \\

\section{Proof of Case (\ref{case1F})}\label{subsection1F}
\setcounter{equation}0

In this case,  we should consider further two subcases:
\begin{eqnarray}
 & M\leq 2^{10} K_1 \,.\label{case11F}   &\\
 & M > 2^{10} K_1\,. \label{case12F} & 
\end{eqnarray}

For the contribution of (\ref{case11F}), 
noticing 
$K\leq CK_1 $ in this subcase, we then estimate (\ref{w1s2}) by
\begin{equation}\label{est-case11F1}
\sum_{K_1\geq K_2\geq K_3} \int_{\mathbb T\times \mathbb R} \left| \left(\sum_{K\leq CK_1}
 \p_x^{s}v_{K}\right) \left( \sum_{K_0\leq CK_2}\p_x u_{K_0} \right)u_{K_1} u_{K_2} u_{K_3} 
 \left(P_{2^{10}K_1}G_{K_3}\right)  \right| dx dt\,,
\end{equation}
which is bounded by
\begin{equation}\label{est-case11F2}
\sum_{K_3}\|u_{K_3}\|_\infty  \| G_{K_3}\|_\infty\int_{\mathbb T\times \mathbb R}
\sum_{K_1} \sum_{K\leq CK_1} K^sv^*_{K} |u_{K_1}|
\sum_{K_2}\sum_{K_0\leq CK_2} K_0 u^*_{K_0}| u_{K_2}| dx dt\,,
\end{equation}
where $f^*$ stands for the Hardy-Littlewood maximal function of $f$. By the Sch\"ur test, 
(\ref{est-case11F2}) can be estimated by 
\begin{equation}\label{est-case11F3}
 \begin{aligned} &\sum_{K_3}K_3^{-\frac{2s-1}{2}} \|u\|_{Y_s}{\mathfrak M}
 \int \left(\sum_{K}|v^*_{K}|^2 \right)^{\frac12} \left(\sum_{K_1}K_1^{2s}|u_{K_1}|^2 \right)^{\frac12} \\
& \hspace{2cm}\cdot \left(\sum_{K_0}K_0|u_{K_0}^*|^2 \right)^{\frac12}  \left(\sum_{K_2}K_2|u_{K_2}|^2 \right)^{\frac12} dx dt\,.
\end{aligned}
\end{equation}
Since $s>1/2$, we then obtain,  by a use of H\"older inequality, that (\ref{est-case11F2}) is
majorized by
\begin{equation}\label{est-case11F4}
\begin{aligned}
& C{\mathfrak M} \|u\|_{Y_s} \left\|\left(\sum_{K}|v^*_{K}|^2 \right)^{\frac12} \right \|_4 
  \left \|\left(\sum_{K_1}K_1^{2s}|u_{K_1}|^2 \right)^{\frac12}   \right\|_4 \\
& \hspace{2cm} \left\| \left(\sum_{K_0}K_0|u_{K_0}^*|^2 \right)^{\frac12}\right\|_4
   \left \|  \left(\sum_{K_2}K_2|u_{K_2}|^2 \right)^{\frac12} \right \|_4\,.
\end{aligned}
\end{equation}
Observe that
\begin{equation}\label{sq1}
\left\|\left(\sum_{K}|v^*_{K}|^2 \right)^{\frac12} \right \|_4 \leq 
 \left\|\left(\sum_{K}|v_{K}|^2 \right)^{\frac12} \right \|_4\leq C\|v\|_4\leq C\| v\|_{0,\frac13}\leq C\,. 
\end{equation}
Here the first inequality is obtained by using Fefferman-Stein's vector-valued inequality on the 
maximal function, and the second one is a consequence of classical Littlewood-Paley theorem.
Similarly, 
\begin{equation}\label{sq2}
\left\|\left(\sum_{K_0}K_0|u^*_{K_0}|^2 \right)^{\frac12} \right \|_4 \leq 
 \left\|\left(\sum_{K_0}K_0|u_{K_0}|^2 \right)^{\frac12} \right \|_4\leq C\|\p_x^{1/2}u\|_4\leq C\| u\|_{\frac12,\frac13}\leq C\|u\|_{Y_s}\,, 
\end{equation}
and 
\begin{equation}\label{sq3}
 \left\|\left(\sum_{K_1}K_1^{2s}|u_{K_1}|^2 \right)^{\frac12} \right \|_4\leq C\|\p_x^{s}u\|_4
 \leq C\| u\|_{s,\frac13}\leq C\|u\|_{Y_s}\,. 
\end{equation}
Hence from (\ref{sq1}), (\ref{sq2}) and (\ref{sq3}), we have
\begin{equation}
(\ref{w1s2})\leq C{\mathfrak M} \|u\|_{Y_s}^4\,. 
\end{equation}

For the contribution of (\ref{case12F}), since in this subcase $K\leq CM$, we estimate 
(\ref{w1s2}) by
\begin{equation}\label{est-case12F1}
\sum_{K_1}\|u_{K_1}\|_\infty   \int_{\mathbb T\times \mathbb R}
\sum_{K_3\leq K_1}|u_{K_3}|\sum_{M} \sum_{K\leq CM} K^sv^*_{K} |g_{K_3, M}|   
\sum_{K_2}\sum_{K_0\leq CK_2} K_0 u^*_{K_0}| u_{K_2}| dx dt\,,
\end{equation}
which is bounded by
\begin{equation}\label{est-case12F2}
\begin{aligned}
 &\sum_{K_1}K_1^{-\frac{2s-1}{2}} \|u\|_{Y_s} \int_{\mathbb T\times \mathbb R}
\sum_{K_3\leq K_1}|u_{K_3}| \left(\sum_K|v^*_K|^2\right)^{1/2}
\left(\sum_{M} M^{2s} |g_{K_3, M}|^2\right)^{1/2}   \\
& \hspace{2cm} \left(\sum_{K_0} K_0| u^*_{K_0}|^2\right)^{1/2} \left(\sum_{K_2}K_2| u_{K_2}|^2\right)^{1/2} dx dt\,.
\end{aligned}
\end{equation}
By a use of Cauchy-Schwarz inequality, (\ref{est-case12F2}) is estimated by
\begin{equation}\label{est-case12F3}
\begin{aligned}
 &\sum_{K_1}K_1^{-\frac{2s-1}{2}}\|u\|_{Y_s}  \int_{\mathbb T\times \mathbb R}
 \left(\sum_K|v^*_K|^2\right)^{1/2}\left(\sum_{K_0} K_0| u^*_{K_0}|^2\right)^{1/2} \left(\sum_{K_2}K_2| u_{K_2}|^2\right)^{1/2}
\\
 & \hspace{2cm}\left(\sum_{K_3}K_3^{2s}|u_{K_3}|^2\right)^{1/2} \left(  \sum_{K_3\leq K_1}\sum_{M} \frac{M^{2s}}{K_3^{2s}} |g_{K_3, M}|^2     \right)^{1/2}   
   dx dt\,.
\end{aligned}
\end{equation}
Using H\"older inequality, we then bound it further by
\begin{equation}\label{est-case12F4}
\begin{aligned}
 &\sum_{K_1}K_1^{-\frac{2s-1}{2}}\|u\|_{Y_s}  
  \left\| \left(\sum_K|v^*_K|^2\right)^{1/2} \right\|_4   
  \left\|  \left(\sum_{K_2} K_0| u^*_{K_0}|^2\right)^{1/2}  \right\|_6
  \left\| \left(\sum_{K_2}K_2| u_{K_2}|^2\right)^{1/2} \right\|_6 \\
 & \hspace{2cm}  \left\| \left(\sum_{K_3}K_3^{2s}|u_{K_3}|^2\right)^{1/2} \right\|_4
 \left\| \left(  \sum_{K_3\leq K_1}\sum_{M} \frac{M^{2s}}{K_3^{2s}} |g_{K_3, M}|^2     \right)^{1/2} \right\|_6  \,,
\end{aligned}
\end{equation}
which is majorized by
\begin{eqnarray*}
 & &\sum_{K_1}K_1^{-\frac{2s-1}{2}}\|u\|_{Y_s}^4\sum_{K_3\leq K_1}K_3^{-s}\left\| \left( \sum_{M} M^{2s}|g_{K_3, M}|^2     \right)^{1/2}\right\|_6 \\
 & \leq & \sum_{K_1}K_1^{-\frac{2s-1}{2}}\|u\|_{Y_s}^4\sum_{K_3\leq K_1}K_3^{-s}\left\| \p_x^{s}G_{K_3}\right\|_\infty \,. 
\end{eqnarray*}

From the definition of $G_{K_3}$ , we have
\begin{equation}\label{estofDG}
  \p_x G_{K_3}(x, t)  
     = {\rm O}\left ( \mathfrak M K_3\right)\|u\|_{Y_s}={\rm O}\left(\mathfrak M K_3 \right)\|\phi\|_{H^s}\,.
\end{equation}
Hence, for $s<1$, 
\begin{equation}\label{psGK}
 \|\p_{x}^sG_{K_3}\|_\infty\leq C\mathfrak M K_3^{s}\|\phi\|_{H^s}\,.
\end{equation}
Since $s>1/2$, we then have
\begin{equation}\label{est-case12F5}
(\ref{est-case12F4})\leq  C\mathfrak M \|\phi\|_{H^s}\sum_{K_1}K_1^{-\frac{2s-1}{2}+\e}\|u\|_{Y_s}^4\leq C\mathfrak M \|\phi\|_{H^s}\|u\|_{Y_s}^4\,.
\end{equation}
This completes our discussion on Case ({\ref{case1F}}).\\

\section{Proof of Case (\ref{case2F})}\label{subsection2F} 
\setcounter{equation}0

In this case, it suffices to consider the following 
subcases:
\begin{eqnarray}
 & K\leq 2^{10}K_2\,;   & \label{subcase21F}\\
  &  K\leq 2^{10}M\,;  & \label{subcase22F}\\
  &  K> 2^{9}(K_2+M)  \,\,\, {\rm and}\,\,\,  K_3 \geq K_0^{1/2}; & \label{subcase23F}\\
  &    K> 2^{9}(K_2+M)\,,  \,\,\,  K_3 \leq K_0^{1/2}\,\,\, {\rm and}\,\,\, M\geq 2^{-10}K_0^{2/3}; & \label{subcase24F}\\
  &  K> 2^{9}(K_2+M)\,, \,\,\, K_3 \leq K_0^{1/2}\,\,\, {\rm and}\,\,\, M < 2^{-10}K_0^{2/3} \,.& 
\label{subcase25F}
\end{eqnarray}

(\ref{subcase21F}) and (\ref{subcase22F}) can be proved exactly the same as  the case (\ref{case11F}) and the case  (\ref{case12F}) respectively.  We omit the details. \\

For the case of (\ref{subcase23F}), observe that (\ref{case2F}) and (\ref{subcase23F}) imply 
\begin{equation}\label{KlessK1}
   K \leq CK_1\,
\end{equation}
and
\begin{equation}\label{K0lessK3}
 K_0^{1/2}\leq K_2^{1/2}K_3^{1/2}\,. 
\end{equation}
Hence (\ref{w1s2}) is bounded by
\begin{equation}\label{subcase23F-est1}
  \int \sum_{K_1}\sum_{K\leq CK_1}K^sv^*_K|u_{K_1}| \sum_{\substack{ K_0\geq K_2\geq K_3\\ K_0\leq K_3^2}}
  K_0u_{K_0}^* |u_{K_2}||u_{K_3}|\|G_{K_3}\|_\infty  dx dt\,.
\end{equation} 
Applying H\"older inequality, we estimate (\ref{subcase23F-est1}) by
\begin{equation}\label{subcase23F-est2}
C\mathfrak M \int \left(\sum_{K}|v^*_K|^2\right)^{\frac12} \left(\sum_{K_1}K_1^{2s}|u_{K_1}|^2\right)^{\frac12}  \prod_{j=0, 2, 3}\left(\sum_{K_j}K_j^{1+\e}|u_{K_j}|^2\right)^{\frac12}
dx dt\,.
\end{equation}
One more use of H\"older inequality yields that (\ref{subcase23F-est1}) is bounded by
$$
C\mathfrak M \left\| \left(\sum_{K}|v_K|^2\right)^{\frac12} \right\|_4
 \left\|\left(\sum_{K_1}K_1^{2s}|u_{K_1}|^2\right)^{\frac12}\right\|_4 
  \prod_{j=0, 2, 3}\left\|\left(\sum_{K_j}K_j^{1+\e}|u_{K_j}|^2\right)^{\frac12}\right\|_6\,.
$$
Hence we obtain 
\begin{equation}
(\ref{subcase23F-est1})\leq C\mathfrak M \|u\|_{Y_s}^4\,.
\end{equation}
This finishes the proof of (\ref{subcase23F}).\\

For the case of (\ref{subcase24F}),  
we estimate (\ref{w1s2}) by
\begin{equation}\label{est1-subcase24F}
\sum_{K_2, K_3} 
 \int\sum_{K_1}\sum_{K\leq CK_1}K^sv^*_K|u_{K_1}| \sum_{ K_0}
  K_0|u^*_{K_0}| |u_{K_2}| |u_{K_3}| \sum_{M\geq  C K_0^{2/3}}\left| g_{K_3, M}\right|  dx dt\,,
\end{equation}
which is dominated by
\begin{equation}\label{est2-subcase24F}
\begin{aligned}
  & \,\,C   \sum_{K_2, K_3}\int \left( \sum_K|v_K^*|^2\right)^{1/2}
 \left( \sum_{K_1}K_1^{2s}|u_{K_1}|^2\right)^{1/2} |u_{K_2}||u_{K_3}| \\
 &\hspace{1.5cm} \left( \sum_{K_0}K_0 |u_{K_0}^*|^2 \right)^{1/2}
 \left(\sum_{M}M^{3/2}|g_{K_3, M}|^{2}\right)^{1/2}  dxdt\,.
\end{aligned}
\end{equation}
By H\"older inequality with $L^4$ norms for the first two functions in the integrand, 
$L^{6+}$ for the next three functions, and $ L^p$ norm (very large $p$) for the last one,   
(\ref{est2-subcase24F}) is dominated by
\begin{equation}
C\|u\|_{Y_s}\sum_{K_2, K_3}\|u_{K_2}\|_{6+}\|u_{K_3}\|_{6+}
    \left\|\left( \sum_{K_0}K_0 |u_{K_0}^*|^2 \right)^{1/2} \right\|_{6+}  
\| \p_x^{3/4}G_{K_3}\|_{\infty} \,.
\end{equation}
Applying (\ref{psGK}), we estimate (\ref{est2-subcase24F}) by
\begin{eqnarray*}
  & & C{\mathfrak M} \|\phi\|_{H_s}\|u\|_{Y_s}^2 \prod_{j=2}^3 \sum_{K_j}K_j^{3/8}\|u_{K_j}\|_{6+} \\
 &\leq & C{\mathfrak M} \|\phi\|_{H_s}\|u\|_{Y_s}^2 \prod_{j=2}^3 \sum_{K_j}K_j^{3/8}\|u_{K_j}\|_{0+, \frac12} \\
& \leq & C{\mathfrak M}\|\phi\|_{H_s}\|u\|_{Y_s}^4 \,,
\end{eqnarray*}
as desired. This completes the discussion of (\ref{subcase24F}).\\

We now turn to the case (\ref{subcase25F}).  In this case, we have
\begin{equation}
 |n_0+n_1| + 2K_2 +M\geq          |n|\geq K/2\geq 2^8(K_2+M)\,,
\end{equation}
which implies 
\begin{equation}\label{largen0+n1}
 |n_0+n_1|\geq 2^5(K_2+M)\,.
\end{equation}
Notice that
\begin{equation}\label{cru-obs}
\begin{aligned}
  &\hspace{2.5cm} (n_0+n_1+n_2+n_3+m)^3-n_0^3-n_1^3-n_2^3-n_3^3-m^3 = \\
  &\hspace{0.5cm} 3(n_0+n_1)(n_0+n_2+n_3+m)(n_1+n_2+n_3+m)
   + (n_2+n_3+m)^3-n_2^3-n_3^3-m^3\,.
   \end{aligned}
 \end{equation}
From (\ref{largen0+n1}), (\ref{cru-obs}) and (\ref{subcase25F}), we obtain
\begin{equation}\label{cru-obs1}
\left| n^3-n_0^3-n_1^3-n_2^3-n_3^3-m^3  \right|\geq C(K_2+M)K_0K_1 \geq CK_0K_1\geq CK_0^2\,.
\end{equation}
Henceforth one of the following four statements must be true:
\begin{eqnarray}
 &  \left| \lambda- n^3\right|\geq  K_0^2 \,, &  \label{IF} \\
&   \left |(\lambda-\lambda_1-\lambda_2-\lambda_3-\mu ) -n_0^3\right|\geq K_0^2\,, &  \label{IIF}\\
&  \exists  i \in \{1, 2, 3\}\,\,\, {\text {such that}}\,\,\, |\lambda_i - n_i^3|\geq  K_0^2\,,& \label{IIIF}\\
& |\mu |\geq K_0^2\,.\label{IVF}
\end{eqnarray}

For the case of (\ref{IF}), we set 
\begin{equation}\label{defofv-ti}
{\ti{ v}}(x, t) = \left({\wh v} \Id_{|\lambda-n^3|\geq K_0^2}\right)^\vee(x, t)\,.
\end{equation} 
We then estimate (\ref{w1s2}) by
\begin{equation}\label{est-IF}
\sum_{K_2, K_3}\|u_{K_2}\|_{\infty}\|u_{K_3}\|_\infty \|G_{K_3}\|_\infty \sum_{K_0}\int |\p_xu_{K_0}| \sum_{K_1}\sum_{K\leq CK_1}K^s{\ti v}_K^*|u_{K_1}| dxdt\,.
\end{equation}
This is clearly bounded by
\begin{equation}\label{est-IF1}
C\mathfrak M \|u\|^2_{Y_s} \sum_{K_0}\int K_0 |u^*_{K_0}| \left( \sum_K |\ti v_K^*|^2\right)^{1/2}
  \left(\sum_{K_1}K_1^{2s}|u_{K_1}|^2\right)^{1/2} dxdt\,.
\end{equation}
Using Cauchy-Schwarz inequality, we bound (\ref{est-IF1}) by
\begin{equation}\label{est-IF2}
C\mathfrak M \|u\|^2_{Y_s} \int\!\! \left(  \sum_{K_0} {K_0^{\e}} |u^*_{K_0}|^2 \right)^{\frac12}\! \left(\sum_{K_0} 
 K_0^{2-\e}\sum_K |\ti v_K^*|^2\right)^{\frac12}\!
  \left(\sum_{K_1}K_1^{2s}|u_{K_1}|^2\right)^{\frac12} \!\!\!\!dxdt\,.
\end{equation}
By H\"older inequality, (\ref{est-IF2}) is majorized by
$$
C\mathfrak M \|u\|^2_{Y_s} \left\|\left(  \sum_{K_0} {K_0^{\e}} |u^*_{K_0}|^2 \right)^{\frac12} 
 \right \|_4 \left\|\left(\sum_{K_0} 
 K_0^{2-\e}\sum_K |\ti v_K^*|^2\right)^{\frac12}\right\|_2
  \left\|\left(\sum_{K_1}K_1^{2s}|u_{K_1}|^2\right)^{\frac12} \right\|_4\,, $$
which is controlled by
\begin{equation}
C\mathfrak  M \|u\|_{Y_s}^3\|\p_x^\e u\|_4\left( \sum_{K_0}K_0^{2-\e}\|\ti v\|_2^2 \right)^{1/2}
\leq C\mathfrak  M \|u\|_{Y_s}^3\|\p_x^\e u\|_4\sum_{K_0}K_0^{-\e/2}\leq  C\mathfrak  M \|u\|_{Y_s}^4 \,.
\end{equation}
This finishes the proof of the case (\ref{IF}). \\

For the case of (\ref{IIF}), let $\ti u$ be defined by
\begin{equation}\label{defoftiu}
{\ti u}= ({\wh u}\Id_{|\lambda-n^3|\geq K_0^2})^\vee\,.
\end{equation}
Then (\ref{w1s2}) can be estimated by
\begin{equation}\label{est-IIF1}
\sum_{K_2, K_3}\|u_{K_2}\|_{\infty}\|u_{K_3}\|_\infty\|G_{K_3}\|_\infty \sum_{K_0}\int |\p_x\ti u_{K_0}| \sum_{K_1}\sum_{K\leq CK_1}K^s{v}_K^*|u_{K_1}| dxdt\,.
\end{equation}
By Sch\"ur test and H\"older inequality, we control (\ref{est-IIF1}) by
\begin{equation}\label{est-IIF2}
\sum_{K_2, K_3}\|u_{K_2}\|_{\infty}\|u_{K_3}\|_\infty\|G_{K_3}\|_\infty 
 \sum_{K_0}\|\p_x\ti u_{K_0}\|_2  \left\| \left( \sum_{K}|v_K|^2 \right)^{1/2}\right\|_4 
\left\| \left( \sum_{K_1}K_1^{2s}|u_{K_1}|^2 \right)^{1/2}\right\|_4\,,
\end{equation}
which is bounded by
\begin{equation}\label{est-IIIF3}
 C\mathfrak M\|u\|_{Y_s}^3\sum_{K_0}\|u_{K_0}\|_{0, \frac12} \leq C\mathfrak M\|u\|_{Y_s}^4\,.
\end{equation}
This completes the proof of the case (\ref{IIF}).\\

For the case of (\ref{IIIF}), if $j=1$,  then we dominate (\ref{w1s2}) by
\begin{equation}\label{est-IIIF1}
\sum_{K_2, K_3}\|u_{K_2}\|_{\infty}\|u_{K_3}\|_\infty\|G_{K_3}\|_\infty \sum_{K_0}\int |\p_x u_{K_0}| \sum_{K_1}\sum_{K\leq CK_1}K^s{v}_K^*|{\ti u}_{K_1}| dxdt\,.
\end{equation}
As we did in the case (\ref{IIF}), we bound (\ref{est-IIIF1}) by
\begin{equation}
C \mathfrak M \|u\|_{Y_s}^2 \sum_{K_0}\|\p_{x}u_{K_0}\|_4\|v\|_4\left\| \left( \sum_{K_1}K_1^{2s}
 |\ti u_{K_1}|^2 \right)^{1/2} \right\|_2\,.
\end{equation}
This can be futher controlled by
\begin{equation}
 C\mathfrak M \|u\|_{Y_s}^3 \sum_{K_0}\frac{1}{K_0}\|\p_{x}u_{K_0}\|_4\|v\|_4
 \leq  C\mathfrak M \|u\|_{Y_s}^3 \sum_{K_0}\frac{1}{K_0}\|u_{K_0}\|_{1, \frac13}\leq 
 C\mathfrak M \|u\|_{Y_s}^4\,,
\end{equation}
as desired.  \\

We now consider $j=2$ or $j=3$. Without loss of generality, assume $j=2$. In this case, we estimate
(\ref{w1s2}) by
\begin{equation}\label{est-IIIF4}
\sum_{K_3}\|u_{K_3}\|\|G_{K_3}\|_\infty \sum_{K_0}\int |\p_x u_{K_0}| \sum_{K_1}\sum_{K\leq CK_1}K^s{v}_K^*|{u}_{K_1}| \sum_{K_2\leq CK_0}|\ti u_{K_2}| dxdt\,,
\end{equation}
which is bounded by
$$
 C\mathfrak M \|u\|_{Y_s}\sum_{K_0}\|\p_xu_{K_0}\|_\infty\sum_{K_2\leq K_0}\|\ti u_{K_2}\|_2 \|v\|_4 \left\| \left( \sum_{K_1}K_1^{2s}
 |u_{K_1}|^2 \right)^{1/2} \right\|_4\,.
$$
Notice that
\begin{eqnarray*}\label{est-IIIF5}
\sum_{K_0}\|\p_xu_{K_0}\|_\infty\sum_{K_2\leq K_0}\|\ti u_{K_2}\|_2 & \leq &
 C\sum_{K_0}\frac{1}{K_0}\|\p_xu_{K_0}\|_\infty \|u\|_{Y_s} \\
 & \leq & C\sum_n\int|\wh u(n, \lambda)|d\lambda\|u\|_{Y_s}  \\
 &\leq &  C\|u\|_{Y_s}^2  \,.
\end{eqnarray*}
Henceforth (\ref{est-IIIF4}) is dominated by
\begin{equation}\label{est-IIIF6}
(\ref{est-IIIF4})\leq C {\mathfrak M}\|u\|_{Y_s}^4\,. 
\end{equation}
This completes the case of (\ref{IIIF}). \\

We now turn to the most difficult case (\ref{IVF}) in Case (\ref{case2F}).  We should decompose $G_{K_3}$, with respect to
the $t$-variable,  into Littlewood-Paley multipliers in the same spirit as before.
  More precisely, for any dyadic number $L$,  let $Q_L$ be 
\begin{equation}\label{defofQL}
 Q_Lu(x, t) = \int \psi_L(\tau)u(x, t-\tau) d\tau\,. 
\end{equation}
Here the Fourier transform of $\psi_L$ is a bump function supported on $[-2L, 2L]$ and $
{\wh {\psi_L}}(x)=1$ if $x\in [-L, L]$. Let 
\begin{equation}\label{defofuL}
\Pi_Lu= Q_Lu - Q_{L/2}u\,.
\end{equation}
Then $\Pi_Lu$ gives a Littlewood-Paley multiplier with respect to the time variable $t$. 
Using this multiplier, we represent
\begin{equation}\label{uKL}
 u_K = \sum_{L}u_{K, L}\,.
\end{equation}
Here $u_{K, L}= \Pi_L(u_K)$. We decompose $G_{K_3}$ as
\begin{equation}\label{decompGK3}
\begin{aligned}
 G_{K_3} & =  C+ \sum_{L}\left( F_3(Q_{L}P_{4K_3}u, \cdots, Q_LP_{K_3/8}u)-F_3(Q_{L/2}P_{4K_3}u, \cdots,
 Q_{L/2}P_{K_3/8}u)  \right)   \\
  & = C+ \sum_{\substack{j=4,2,1,\frac12, \frac14, \frac18\\L }}
           H_{K_3, L}\,u_{jK_3, L}\,,  
\end{aligned}
\end{equation} 
where $H_{K_3, L}$ is given by
\begin{equation}\label{defofHK3}
H_{K_3, L}= F_4\left(Q_{\ell L}P_{4K_3}u, \cdots, Q_{\ell L}P_{K_3/8}u; \ell=1, \frac12\right)\,.
\end{equation}

Let ${\mathfrak M}_1 $ be defined by
\begin{equation}
\mathfrak M_1=\sup \left\{ |D^\alpha F_4(u_1,\cdots, u_{12})|  :  u_j\,\,\,{\text{satisfies}} \,\,\, \|u_j\|_{Y_s}\leq C_0\|\phi\|_{H^s}\,\,\, {\text{for all} }\,\,\, j=1, \cdots, 12;\,\,\, \alpha \right\}\,.
\end{equation}
Here $ D^\alpha=\p_{x_1}^{\alpha_1}\cdots\p_{x_{12}}^{\alpha_{12}} $ and $\alpha $ is taken over all tuples $ (\alpha_1, \cdots, \alpha_{12})\in (\mathbb N \cup \{0\})^{12}$ with 
$0\leq \alpha_j\leq 1$ for all $j\in\{1, \cdots, 12\}$.
$\mathfrak M_1$ is a real number because $F_4\in C^1$. \\

In order to finish the proof, we need to consider further three subcases:
\begin{eqnarray}
  & L\leq 2^{10}K_3^3 \,,  & \label{LK3}\\
  &  2^{10}K_3^3<L\leq 2^{-5}K_0^2\,, & \label{K3L}\\
  &  L> 2^{-5} K_0^2\,. & \label{LK0}
\end{eqnarray}

For the contribution of (\ref{LK3}), 
we set
\begin{equation}\label{hKL}
h_{K_0, jK_3, L}= \left( \wh{H_{K_3,L}u_{jK_3, L}}\Id_{|\mu|\geq K_0^2}\right)^\vee\,.
\end{equation}
Here $j=4, 2, 1, \frac12, \frac14, \frac18$. 
From the definition of $H_{K_3, L}$, we get
\begin{equation}\label{L4ofHK3}
\left\| h_{K_0, jK_3, L}\right\|_4
\leq C{\mathfrak M_1}\|\phi\|_{H^s}\frac{L}{K_0^2}\|u_{jK_3, L}\|_4\,.
\end{equation}
Then (\ref{w1s2}) is bounded by
\begin{equation}\label{est-LK3}
\begin{aligned}
  & \sum_{K_2}\|u_{K_2}\|_{\infty}  \sum_{K_0}\int
 K_0 u^*_{K_0} \sum_{K_3\leq CK_0^{1/2}}\|u_{K_3}\|_\infty \\
 &
\hspace{1cm}\cdot \sum_{L\leq CK_3^3}\left|h_{K_0, jK_3, L}\right|
 \sum_{K_1}\sum_{K\leq CK_1}K^s{v}_K^*|{ u}_{K_1}| dxdt\,,
\end{aligned}
\end{equation}
which is majorized by 
\begin{equation}\label{est-LK3-1}
\begin{aligned}
  & \sum_{K_2}\|u_{K_2}\|_{\infty}  \sum_{K_0}K_0\sum_{K_3\leq CK_0^{1/2}}\|u_{K_3}\|_\infty\int
  u^*_{K_0}  \\
 &
\hspace{1cm}\cdot \sum_{L\leq CK_3^3}\left|h_{K_0, jK_3, L}\right|
\left( \sum_K |v_K^*|^2\right)^{1/2} 
\left(\sum_{K_1}K_1^{2s}|u_{K_1}|^2\right)^{1/2} dxdt\,.
\end{aligned}
\end{equation}

Using H\"older inequality with $L^4$ norms for four functions in the integrand,  we estimate (\ref{est-LK3-1}) by 
\begin{equation}\label{est-LK31}
\begin{aligned}
 &    C{\mathfrak M_1}\|\phi\|_{H^s} \|u\|^2_{Y_s}\sum_{K_0}K_0 \| u_{K_0}\|_4
 \sum_{K_3\leq K_0^{1/2}}\|u_{K_3}\|_\infty 
\sum_{L\leq CK_3^3}\frac{L}{K_0^2}\left\|u_{jK_3, L} \right\|_4 \\
\leq &\,  C{\mathfrak M_1}\|\phi\|_{H^s}^2 \|u\|^3_{Y_s}\sum_{K_0}K_0^{1/2} \| u_{K_0}\|_{0, \frac13} \\
\leq &\, C{\mathfrak M_1}\|\phi\|_{H^s}^2 \|u\|^4_{Y_s} \,.
\end{aligned}
\end{equation}
This finishes the case of (\ref{LK3}).\\

For the contribution of (\ref{K3L}), 
we bound (\ref{w1s2}) by
\begin{equation}\label{est-K3L}
\begin{aligned}
  & \sum_{K_2}\|u_{K_2}\|_{\infty}\sum_{K_3} \|u_{K_3}\|_\infty
 \int \sum_{K_0} |\p_x u_{K_0}| \sum_{2^{10}K_3^3<L\leq 2^{-10}K_0^2}|h_{K_0, jK_3, L}|\\
 &
\hspace{1cm}\cdot  \sum_{K_1}\sum_{K\leq CK_1}K^s{v}_K^*|{ u}_{K_1}| dxdt\,,
\end{aligned}
\end{equation}
which is dominated by
\begin{equation}\label{est-K3L1}
\begin{aligned}
  & C\|u\|_{Y_s}\sum_{K_3} \|u_{K_3}\|_\infty \sum_{\substack{\Delta\leq 2^{-10}\\
 \Delta \,\,{\rm dyadic}}}
 \int \sum_{K_0} |\p_x u_{K_0}| \sum_{\substack{2^{10}K_3^3<L\\
   \frac{\Delta}{2}K_0^2< L\leq \Delta K_0^2}}|h_{K_0, jK_3, L}|\\
 &
\hspace{1cm}\cdot  \left( \sum_{K}|v^*_K|^2 \right)^{1/2}
  \left( \sum_{K_1}K_1^{2s}|u_{K_1}|^2\right)^{1/2} dxdt\,,
\end{aligned}
\end{equation}
By Cauchy-Schwarz inequality, we estimate (\ref{est-K3L1}) further by
\begin{equation}\label{est-K3L2}
\begin{aligned}
  & C\|u\|_{Y_s}\sum_{K_3} \|u_{K_3}\|_\infty \sum_{\substack{\Delta\leq 2^{-10}\\
 \Delta \,\,{\rm dyadic}}} \Delta^{-1/2}
 \int \sum_{K_0} \frac{|\p_x u_{K_0}|}{K_0}
  \\
 &
\hspace{1cm}\cdot \left(\sum_{\substack{2^{10}K_3^3<L\\
   \frac{\Delta}{2}K_0^2< L\leq \Delta K_0^2}}L|h_{K_0, jK_3, L}|^2\right)^{1/2} \left( \sum_{K}|v^*_K|^2 \right)^{1/2}
  \left( \sum_{K_1}K_1^{2s}|u_{K_1}|^2\right)^{1/2} dxdt\,,
\end{aligned}
\end{equation}
Applying H\"older inequality with $L^\infty$ norm for the first function in the integrand, 
 $ L^2$ norm for the second one, and $L^4$ norms for the last two functions,   
 we then majorize (\ref{est-K3L2}) by
\begin{equation}\label{est-K3L3}
   C\|u\|_{Y_s}^2\sum_{K_3} \|u_{K_3}\|_\infty\!\!\! \sum_{\substack{\Delta\leq 2^{-10}\\
 \Delta \,\,{\rm dyadic}}} \Delta^{-1/2}
 \sum_{K_0} \frac{\|\p_x u_{K_0}\|_\infty}{K_0}
  \left\| \left(\sum_{\substack{2^{10}K_3^3<L\\
   \frac{\Delta}{2}K_0^2< L\leq \Delta K_0^2}}L|h_{K_0, jK_3, L}|^2\right)^{1/2}\right\|_2 \,.
\end{equation}
Notice that if $L\sim \Delta K_0^2$, then
\begin{equation}\label{L2hKL}
 \|h_{K_0, jK_3, L}\|_2\leq C{\mathfrak M_1}\|\phi\|_{H^s}\Delta \|u_{jK_3, L}\|_2\,.
\end{equation}
Thus we have
\begin{equation}\label{L2hKL1}
\begin{aligned}
  &\left\| \left(\sum_{\substack{2^{10}K_3^3<L\\
   \frac{\Delta}{2}K_0^2< L\leq \Delta K_0^2}}L|h_{K_0, jK_3, L}|^2\right)^{1/2}\right\|_2 \\
\leq & \,C{\mathfrak M_1}\|\phi\|_{H^s}\Delta \left( \sum_{\substack{2^{10}K_3^3<L\\
   \frac{\Delta}{2}K_0^2< L\leq \Delta K_0^2}}L \|u_{jK_3, L}\|_2^2  \right)^{1/2}\\
 \leq  & \,C{\mathfrak M_1}\|\phi\|_{H^s}\Delta \|u_{jK_3}\|_{0, \frac12}
 \\\leq &\,C{\mathfrak M_1}\|\phi\|^2_{H^s}\Delta  \,.
\end{aligned}
\end{equation}
From (\ref{L2hKL1}), (\ref{est-K3L3}) is bounded by
\begin{equation}\label{est-K3L4}
   C{\mathfrak M_1}\|\phi\|^2_{H^s}\|u\|_{Y_s}^2
  \sum_{K_3} \|u_{K_3}\|_\infty\!\!\! \sum_{\substack{\Delta\leq 2^{-10}\\
 \Delta \,\,{\rm dyadic}}} \Delta^{1/2}
 \sum_{K_0} \frac{\|\p_x u_{K_0}\|_\infty}{K_0}\,,
\end{equation}
which is clearly majorized by 
\begin{equation}\label{est-K3L5}
   C{\mathfrak M_1}\|\phi\|^2_{H^s}\|u\|_{Y_s}^4\,.
\end{equation}
This finishes the case of (\ref{K3L}).\\

For the contribution of (\ref{LK0}), we estimate (\ref{w1s2}) by
\begin{equation}\label{est-LK0-1}
\begin{aligned}
  & \sum_{K_2}\|u_{K_2}\|_{\infty}\sum_{K_3} \|u_{K_3}\|_\infty
 \int \sum_{K_0} |\p_x u_{K_0}| \sum_{L > 2^{-5}K_0^2}|h_{K_0, jK_3, L}|\\
 &
\hspace{1cm}\cdot  \sum_{K_1}\sum_{K\leq CK_1}K^s{v}_K^*|{u}_{K_1}| dxdt\,,
\end{aligned}
\end{equation}
which is bounded by
\begin{equation}\label{est-LK0-2}
\begin{aligned}
  & \sum_{K_2}\|u_{K_2}\|_{\infty}\sum_{K_3} \|u_{K_3}\|_\infty
 \int \left(\sum_{K_0}\frac{ |\p_x u_{K_0}|^2}{K_0^2}\right)^{1/2}
  \left( \sum_{L > 2^{-5}K_0^2}L|h_{K_0, jK_3, L}|^2\right)^{1/2}\\
 &
\hspace{1cm}\cdot \left( \sum_K |v_K^*|^2 \right)^{1/2}
       \left( \sum_{K_1}K_1^{2s}|u_{K_1}|^2\right)^{1/2}\,.
\end{aligned}
\end{equation}
Applying H\"older inequality, we estimate (\ref{est-LK0-2}) further by
\begin{equation}\label{est-LK0-3}
   C{\mathfrak M}_1\|u\|_{Y_s}^2\sum_{K_3} \|u_{K_3}\|_\infty
 \sum_{K_0} \frac{\|\p_x u_{K_0}\|_\infty}{K_0}
 \left\|\left( \sum_{L > 2^{-5}K_0^2}L|u_{ jK_3, L}|^2\right)^{1/2}\right\|_2 \,.
\end{equation}
This is clearly majorized by 
\begin{equation}
C{\mathfrak M}_1 \|\phi\|_{H_s}\|u\|_{Y_s}^4\,.
\end{equation}
Hence we complete the case of (\ref{LK0}).\\


\section{Proof of Case (\ref{case3F})}\label{subsection3F} 
\setcounter{equation}0

In this case, it suffices to consider the following subcases:
\begin{eqnarray}
 & M\geq 2^{-10}K_0^{2/3}\,;   & \label{subcase31F}\\
  &  M  < 2^{-10}K_0^{2/3}\,\,\, {\rm and}\,\,\, K_2^2K_3\geq 2^{-10}K_0^2;  & \label{subcase32F}\\
  &  M  < 2^{-10}K_0^{2/3}\,\,\, {\rm and}\,\,\, K_2^2M\geq 2^{-10}K_0^2;     & \label{subcase33F}\\
  &  M  < 2^{-10}K_0^{2/3}\,,\,\,\, K_2^2K_3 < 2^{-10}K_0^2 \,\,\,{\rm and}\,\,\, K_2^2M <2^{-10}K_0^2
 \,. & \label{subcase34F} 
\end{eqnarray}

For the case of (\ref{subcase31F}), notice that, in this case, we have
\begin{equation}\label{KlessM}
 K\leq CM^{3/2}\,. 
\end{equation}
Henceforth we estimate (\ref{w1s2}) by
\begin{equation}\label{est1-case31F}
 \int\sum_{K_1\geq K_2\geq K_3}|u_{K_1}||u_{K_2}||u_{K_3}|
   \sum_M \sum_{K\leq CM^{3/2}}K^sv^*_K\sum_{K_0\leq CM^{3/2}}K_0u_{K_0}^*
    |g_{K_3, M}|
dxdt\,,
\end{equation}
which is bounded by
\begin{equation}\label{est2-case31F}
  \int\sum_{K_1\geq K_2\geq K_3}|u_{K_1}||u_{K_2}||u_{K_3}|
   \sum_M  M^{\frac{3}{2}(1-s)} |g_{K_3, M}|\sum_{K\leq CM^{3/2}}K^sv^*_K  
    \left(\sum_{K_0}K_0^{2s}|u_{K_0}^*|^2\right)^{1/2}
  dxdt\,,
\end{equation}
since $ 1/2<s<1$. Applying Sch\"ur test, we estimate (\ref{est2-case31F}) by
\begin{equation}\label{est3-case31F}
\begin{aligned}
 &  \int\sum_{K_1\geq K_2\geq K_3}|u_{K_1}||u_{K_2}||u_{K_3}|
   \left(\sum_M  M^{3} |g_{K_3, M}|^2 \right)^{1/2}\\
 & \hspace{1cm}\left(\sum_{K}|v^*_K|^2\right)^{1/2}  
    \left(\sum_{K_0}K_0^{2s}|u_{K_0}^*|^2\right)^{1/2}
  dxdt\,.
\end{aligned}
\end{equation}
By H\"older inequality and $s>1/2$, (\ref{est3-case31F}) is majorized by
\begin{equation}\label{est4-case31F}
\begin{aligned}
& C\sum_{K_1\geq K_2\geq K_3}\|\p^{3/2}_xG_{K_3}\|_\infty \left(\prod_{j=1}^3\|u_{K_j}\|_{6+}\right)
\left\|\left(\sum_{K}|v_K|^2\right)^{1/2}  \right\|_4\left\| \left(\sum_{K_0}K_0^{2s}|u_{K_0}^*|^2\right)^{1/2} \right\|_4\\
 \leq & \,C\mathfrak{M}(\|\phi\|_{H^s}+\|\phi\|_{H^s}^2)\|u\|_{Y_s}\sum_{K_1\geq K_3\geq K_3}K_3^{3/2}
 \prod_{j=1}^3\|u_{K_j}\|_{6+} \\
\leq & \,C\mathfrak{M}(\|\phi\|_{H^s}+\|\phi\|_{H^s}^2)\|u\|_{Y_s}\prod_{j=1}^3
 \sum_{K_j}K_j^{1/2}\|u_{K_j}\|_{0+, \frac12} \\
\leq & \,C\mathfrak{M}(\|\phi\|_{H^s}+\|\phi\|_{H^s}^2)\|u\|_{Y_s}^4  \,.
\end{aligned}
\end{equation}
This finishes the case of (\ref{subcase31F}).  \\

For the case of (\ref{subcase32F}),  observe that, in this case, 
\begin{equation}\label{shift32F}
 K_0\leq CK_1^{1/2}K_2^{1/2}K_3^{1/2}\,.
\end{equation}
We estimate (\ref{w1s2}) by
\begin{equation}\label{est1-case32F}
 \int\sum_{K_1\geq K_2\geq K_3}|u_{K_1}||u_{K_2}||u_{K_3}|
   \sum_{K\leq CK_0}K^sv^*_K\sum_{K_0\leq C(K_1K_2K_3)^{1/2}}K_0u_{K_0}^*
    \|G_{K_3}\|_\infty
dxdt\,,
\end{equation}
which is bounded by
\begin{equation}\label{est2-case32F}
C\mathfrak M \int \left(\sum_{K}|v^*_K|^2\right)^{1/2}  
    \left(\sum_{K_0}K_0^{2s}|u_{K_0}^*|^2\right)^{1/2}\prod_{j=1}^3\sum_{K_j}K_j^{1/2}|u_{K_j}| dxdt\,.
\end{equation}
Using H\"older inequality with $L^4$ norms for first two functions and $L^6$ norms for 
the last three functions in the integrand, we obtain
\begin{equation}\label{est3-case32F}
C\mathfrak M \|u\|_{Y_s} \prod_{j=1}^3\left\|\sum_{K_j}K_j^{1/2}|u_{K_j}|\right\|_{6}
\leq  C\mathfrak M \|u\|_{Y_s}^4\,.
\end{equation}
This completes the case of (\ref{subcase32F}).\\

For the case of (\ref{subcase33F}),  we have, in this case, 
\begin{equation}\label{shift33F}
 K_0\leq CK_1^{1/2}K_2^{1/2}M^{1/2}\,.
\end{equation}
Hence we dominate (\ref{w1s2}) by
\begin{equation}\label{est1-case33F}
 \int\sum_{K_1\geq K_2\geq K_3}|u_{K_1}||u_{K_2}||u_{K_3}|
   \sum_M|g_{K_3, M}|\sum_{K\leq CK_0}K^sv^*_K\sum_{K_0\leq C(K_1K_2M)^{1/2}}K_0u_{K_0}^*
dxdt\,,
\end{equation}
which is bounded by
\begin{equation}\label{est2-case33F}
\begin{aligned}
 & \,C \sum_{K_3}\int \left(\sum_{K}|v^*_K|^2\right)^{1/2}  
    \left(\sum_{K_0}K_0^{2s}|u_{K_0}^*|^2\right)^{1/2}|u_{K_3}|\\
  & \cdot \left(\sum_{M} M|g_{K_3, M}|^2\right)^{1/2}
\prod_{j=1}^2\sum_{K_j}K_j^{1/2}|u_{K_j}| dxdt\,.
\end{aligned}
\end{equation}
Using H\"older inequality with $L^4$ norms for first two functions, $L^6$ norms for
the third one,  $L^{p}$ norm with $p$ very large for the fourth one, and 
$L^{6+}$ for the last two functions in the integrand, we obtain
\begin{equation}\label{est3-case34F}
C\|u\|_{Y_s} \prod_{j=1}^2\left\|\sum_{K_j}K_j^{1/2}|u_{K_j}|\right\|_{6+}
\sum_{K_3}\|u_{K_3}\|_6 \|\p_x^{1/2}G_{K_3}\|_\infty \,.
\end{equation}
Clearly (\ref{est3-case34F}) is dominated by
\begin{equation}\label{est4-case34F}
C\mathfrak M\|\phi\|_{H^s}\|u\|_{Y_s}^3 \sum_{K_3}K_3^{1/2}\|u_{K_3}\|_6
 \leq C\mathfrak M\|\phi\|_{H^s}\|u\|_{Y_s}^4 \,.
\end{equation}
Hence the case of (\ref{subcase33F}) is done.\\

For the case of (\ref{subcase34F}), we observe that, in this case,
\begin{equation}\label{MlargeK2}
M^2K_2\leq 2^{-10}K_0^2\,. 
\end{equation}
In fact, if (\ref{MlargeK2}) does not hold, then from (\ref{subcase34F}), 
$$
 M^2K_2 > 2^{-10}K_0^2 > K_2^2M\,.
$$
Thus $M> K_2$, which yields immediately
$$
 M^3 > M^2K_2 > 2^{-10}K_0^2\,,
$$
contradicting to $ M < 2^{-10}K_0^{2/3}$. Hence (\ref{MlargeK2}) must be true. 
From (\ref{MlargeK2}), $ K_2^2K_3 +K_2^2M < 2^{-9}K_0^2$, we get
\begin{equation}\label{thesmall}
 \left|(n_2+n_3+m)^3-n_2^3-n_3^3-m^3\right|\leq 2^{-5}K_0^2\,. 
\end{equation} 
Since $n_1+n_2+n_3+m\neq 0$, from (\ref{case3F}),  (\ref{subcase34F}) and (\ref{thesmall}), 
the crucial arithmetic observation (\ref{cru-obs}) then yields 
\begin{equation}\label{cru-est3F}
 |n^3-n_0^3-n_1^3-n_2^3-n_3^3-m^3| \geq 2K_0^2\,. 
\end{equation}

Henceforth one of the following four statements must be true:
\begin{eqnarray}
 &  \left| \lambda- n^3\right|\geq  K_0^2 \,, &  \label{I3F} \\
&   \left |(\lambda-\lambda_1-\lambda_2-\lambda_3-\mu ) -n_0^3\right|\geq K_0^2\,, &  \label{II3F}\\
&  \exists  i \in \{1, 2, 3\}\,\,\, {\text {such that}}\,\,\, |\lambda_i - n_i^3|\geq  K_0^2\,,& \label{III3F}\\
& |\mu |\geq K_0^2\,.\label{IV3F}
\end{eqnarray}

For the case of (\ref{I3F}),  
we estimate (\ref{w1s2}) by
\begin{equation}\label{est-I3F}
\sum_{K_1, K_2, K_3}\|u_{K_1}\|_\infty\|u_{K_2}\|_{\infty}\|u_{K_3}\|_\infty \|G_{K_3}\|_\infty 
 \sum_{K_0}\int K_0|u^*_{K_0}| \left| \sum_{K\leq CK_0} {\p_x^s\ti v}_K \right| dxdt\,.
\end{equation}
Then Cauchy-Schwarz inequality yields
\begin{equation}\label{est-I3F1}
\begin{aligned}
  & C\mathfrak M \|u\|^3_{Y_s} \left\|
  \left( \sum_{K_0}K_0^{2-2s}\left|\sum_{K\leq CK_0} \p_x^s {\ti v}_K \right|^2 \right)^{1/2} \right\|_2
\left\| \left( \sum_{K_0} K_0^{2s} |u^*_{K_0}|^2\right)^{1/2}\right\|_2\\
 \leq  &\, C\mathfrak M \|u\|^4_{Y_s} \left(\sum_{K_0}K_0^{2-2s}\sum_{K\leq CK_0}\left \|\p_x^s {\ti v}_K 
 \right\|_2^2   \right)^{1/2}   \,\leq \, C\mathfrak M \|u\|^4_{Y_s}\,.
\end{aligned}
\end{equation}
This finishes the proof of the case (\ref{I3F}). \\

For the case of (\ref{II3F}), 
(\ref{w1s2}) can be estimated by
\begin{equation}\label{est-II3F1}
\sum_{K_1, K_2, K_3}\|u_{K_1}\|_\infty\|u_{K_2}\|_{\infty}\|u_{K_3}\|_\infty\|G_{K_3}\|_\infty \sum_{K_0}\int K_0|\ti u_{K_0}^*| \sum_{K\leq CK_0}K^s{v}_K^* dxdt\,.
\end{equation}
By Sch\"ur test and H\"older inequality, we control (\ref{est-II3F1}) by
\begin{equation}\label{est-II3F2}
C\mathfrak M\|u\|_{Y_s}^3  \left\| \left( \sum_{K}|v^*_K|^2 \right)^{1/2}\right\|_2
\left\| \left( \sum_{K_0}K_0^{2s+2}|\ti u_{K_0}|^2 \right)^{1/2}\right\|_2\,,
\end{equation}
which is clearly  bounded by
\begin{equation}\label{est-III3F3}
 C\mathfrak M\|u\|_{Y_s}^3\left( \sum_{K_0} K_0^{2s}\|u_{K_0}\|_{0, \frac12}^2 \right)^{1/2}\,\leq\, C\mathfrak M\|u\|_{Y_s}^4\,.
\end{equation}
This completes the proof of the case (\ref{II3F}).\\

For the case of (\ref{III3F}), without loss of generality, assume 
$j=1$.  We then dominate (\ref{w1s2}) by
\begin{equation}\label{est-III3F1}
\sum_{K_2, K_3}\|u_{K_2}\|_{\infty}\|u_{K_3}\|_\infty\|G_{K_3}\|_\infty \sum_{K_1}\sum_{K_0}\int  K_0| u^*_{K_0}| |{\ti u}_{K_1}| \sum_{K\leq CK_0}K^s{v}_K^* dxdt\,.
\end{equation}
By H\"older inequality, we bound (\ref{est-III3F1}) by
\begin{equation}\label{est1-III3F1}
\begin{aligned}
 & \sum_{K_2, K_3}\|u_{K_2}\|_{\infty}\|u_{K_3}\|_\infty\|G_{K_3}\|_\infty \sum_{K_1}\sum_{K_0}\sum_{K\leq CK_0} K^s K_0 \|u_{K_0}\|_4 \|{\ti u}_{K_1}\|_2 \|v_K\|_4\\
 \leq &\, \sum_{K_2, K_3}\|u_{K_2}\|_{\infty}\|u_{K_3}\|_\infty\|G_{K_3}\|_\infty \sum_{K_1}
   \| u_{K_1}\|_{0, \frac12}\sum_{K_0}\sum_{K\leq CK_0} K^s \|u_{K_0}\|_4  \|v_K\|_4 \,.
\end{aligned}
\end{equation}
By Sch\"ur test, we dominate (\ref{est1-III3F1}) by
\begin{equation}\label{est2-III3F1}
\begin{aligned}
  &\,   C\mathfrak M \|u\|_{Y_s}^2\sum_{K_1}
   \| u_{K_1}\|_{0, \frac12}\left(\sum_{K_0} K_0^{2s} \|u_{K_0}\|^2_4\right)^{1/2}
   \left( \sum_{K}\|v_K\|^2_4\right)^{1/2} \\
 \leq  & \,C\mathfrak M \|u\|_{Y_s}^3\left(\sum_{K_0} K_0^{2s} \|u_{K_0}\|^2_{0, \frac13}\right)^{1/2}
   \left( \sum_{K}\|v_K\|^2_{0, \frac13}\right)^{1/2} \\
\leq &\, C\mathfrak M \|u\|_{Y_s}^4  \,.
\end{aligned}
\end{equation}
Hence the case of (\ref{III3F}) is done. \\

In order to finish the proof, as before we need to consider further three subcases:
\begin{eqnarray}
  & L\leq 2^{10}K_3^3 \,,  & \label{LK3-3}\\
  &  2^{10}K_3^3<L\leq 2^{-5}K_0^2\,, & \label{K3L-3}\\
  &  L> 2^{-5} K_0^2\,. & \label{LK0-3}
\end{eqnarray}

For the contribution of (\ref{LK3-3}), 
notice that
\begin{equation}\label{L6ofHK3}
 \|h_{K_0, jK_3, L}\|_6\leq C\mathfrak M_1\|\phi\|_{H^s}\frac{L}{K_0^2}\|u_{jK_3, L}\|_{6}\,.
\end{equation}
Here $h_{K_0, jK_3, L}$ is defined as in (\ref{hKL}). In this particular case we also have 
$K_3\leq K_0^{2/3}$ from $K_2^2K_3\leq 2^{-10}K_0^2$. 
Then (\ref{w1s2}) is bounded by
\begin{equation}\label{est-LK3-3}
  \int\sum_{K_0}K_0u^*_{K_0}\sum_{K\leq CK_0}K^s{v}_K^* \sum_{\substack{K_1\geq K_2 \geq K_3 \\ K_3\leq K_0^{2/3}}} |u_{K_1}||u_{K_2}||u_{K_3}| \sum_{L\leq CK_3^3}\left|h_{K_0, jK_3, L}\right| 
  dxdt\,.
\end{equation}
Write (\ref{est-LK3-3}) as
\begin{equation}\label{est1-LK3-3}
  \sum_{\substack{\Delta \,{\rm dyadic}\\ \Delta \leq 1 }}\int\sum_{K_0}K_0u^*_{K_0}
  \!\!\!\sum_{K\leq CK_0}K^s{v}_K^* \!\!\!\!\!\!\!\!\!\!\!\!\!\!\sum_{\substack{K_1\geq K_2 \geq K_3 \\    \Delta K_0^{2/3}/2 < K_3\leq \Delta K_0^{2/3}}} 
 \!\!\!\!\!\!\!\!\!\!\!|u_{K_1}||u_{K_2}||u_{K_3}| \sum_{L\leq CK_3^3}\left|h_{K_0, jK_3, L}\right| 
  dxdt\,.
\end{equation}
Observe that if $ \Delta K_0^{2/3}/2 < K_3\leq \Delta K_0^{2/3} $, then we have
\begin{equation}\label{shift-F}
 K_0\leq \Delta^{-3/2}K_1^{1/2}K_2^{1/2}K_3^{1/2}\,. 
\end{equation}
Henceforth, (\ref{est1-LK3-3}) is bounded by
\begin{equation}\label{est2-LK3-3}
\begin{aligned}
 &  C\|u\|_{Y_s}\sum_{K_0}\sum_{K\leq K_0} K^s
 \sum_{K_1, K_2}K_1^{1/2}K_2^{1/2}\sum_{\Delta\leq 1}\Delta^{-3/2} 
\sum_{K_3\sim \Delta K_0^{2/3}} K_3^{1/2} \\
& \int u_{K_0}^* v^*_{K}|u_{K_1}||u_{K_2}|\sum_{L\leq CK^3_3}|h_{K_0, jK_3, L}|dxdt  
\,.
\end{aligned}
\end{equation}
Applying H\"older inequality with $L^4$ norms for first two functions and 
$L^6$ for the last three functions, and then using (\ref{L6ofHK3}), we get
\begin{equation}\label{est3-LK3-3}
\begin{aligned}
 & C\mathfrak M_1\|\phi\|_{H^s}\|u\|_{Y_s}\sum_{K_0}\sum_{K\leq K_0} K^s
 \sum_{K_1, K_2}K_1^{1/2}K_2^{1/2}\sum_{\Delta\leq 1}\Delta^{-3/2}
\sum_{K_3\sim\Delta K_0^{2/3}} K^{1/2}_3\\
 & \hspace{1cm} \|u_{K_0}\|_4 \| v^*_{K}\|_4
 \|u_{K_1}\|_6\|u_{K_2}\|_6 \sum_{L\leq CK_3^3}\frac{L}{K_0^2}\|u_{jK_3, L}\|_{6}  
\,,
\end{aligned}
\end{equation}
which is bounded by
\begin{equation}\label{est4-LK3-3}
\begin{aligned}
 & C\mathfrak M_1\|\phi\|_{H^s}\|u\|_{Y_s}\sum_{K_0}\sum_{K\leq K_0} K^s
 \sum_{\Delta\leq 1}\Delta^{-3/2}
 \sum_{L\leq C\Delta^3  K_0^2}\frac{L}{K_0^2}
\\
 & \hspace{0.5cm} \|u_{K_0}\|_4 \| v^*_{K}\|_4
\sum_{K_1}K_1^{1/2}\|u_{K_1}\|_{0+, \frac12} 
 \sum_{K_2}K_2^{1/2}\|u_{K_2}\|_{0+, \frac12}  \sum_{K_3} K^{1/2}_3\|u_{jK_3, L}\|_{0+, \frac12}\\
 \leq &\, C\mathfrak M_1\|\phi\|_{H^s}^2\|u\|_{Y_s}^3 
   \sum_{\Delta\leq 1}\Delta^{3/2}\sum_{K_0}\sum_{K\leq CK_0}K^s\|u_{K_0}\|_4\|v_{K}\|_4  \\
 \leq &\, C\mathfrak M_1\|\phi\|_{H^s}^2\|u\|_{Y_s}^3 
    \left(\sum_{K_0}K_0^{2s}\|u_{K_0}\|_{0, \frac13}^2\right)^{1/2}
   \left(\sum_{K}\|v_{K}\|_{0, \frac13}^2\right)^{1/2}\\
  \leq &\, C\mathfrak M_1\|\phi\|_{H^s}^2\|u\|_{Y_s}^4 \,. 
\end{aligned}
\end{equation}
This completes the case (\ref{LK3-3}).\\

For the contribution of (\ref{K3L-3}),  (\ref{w1s2}) is bounded by
\begin{equation}\label{est-K3L-3}
\begin{aligned}
  & \sum_{K_1}\|u_{K_1}\|_\infty\sum_{K_2}\|u_{K_2}\|_{\infty}\sum_{K_3} \|u_{K_3}\|_\infty
 \int \sum_{K_0}\sum_{K\leq CK_0}K^s{v}^*_K K_0u^*_{K_0}\\
 &  \sum_{2^{10}K_3^3<L\leq 2^{-5}K_0^2}|h_{K_0, jK_3, L}|
                 dxdt\,,
\end{aligned}
\end{equation}
which is dominated by
\begin{equation}\label{est-1-K3L-3}
 C\|u\|^2_{Y_s}\sum_{K_3} \|u_{K_3}\|_\infty \sum_{\substack{\Delta\leq 2^{-5}\\
 \Delta \,\,{\rm dyadic}}}\sum_{K_0}\sum_{K\leq CK_0}K^s
  \int K_0 u^*_{K_0} v_K^*\!\!\!\!\!\! \sum_{\substack{2^{10}K_3^3<L\\
   \frac{\Delta}{2}K_0^2< L\leq \Delta K_0^2}}|h_{K_0, jK_3, L}|dxdt\,.
\end{equation}
Using Cauchy-Schwarz inequality, we estimate (\ref{est-1-K3L-3}) further by
\begin{equation}\label{est-2-K3L-3}
\begin{aligned}
 &  C\|u\|^2_{Y_s}\sum_{K_3} \|u_{K_3}\|_\infty \sum_{\substack{\Delta\leq 2^{-5}\\
 \Delta \,\,{\rm dyadic}}} \Delta^{-\frac12}\sum_{K_0}\sum_{K\leq CK_0}K^s
  \int  u^*_{K_0} v_K^*\\
 &\hspace{2cm}\left( \sum_{\substack{2^{10}K_3^3<L\\
   \frac{\Delta}{2}K_0^2< L\leq \Delta K_0^2}}L|h_{K_0, jK_3, L}|^2\right)^{1/2}dxdt\,.
\end{aligned}
\end{equation}
Employing H\"older inequality with $L^4$ norms for the first two functions and 
$L^2$ for the last one, we bound (\ref{est-2-K3L-3}) by
\begin{equation}\label{est-3-K3L-3}
\begin{aligned}
 &   C\|u\|^2_{Y_s}\sum_{K_3} \|u_{K_3}\|_\infty \sum_{\substack{\Delta\leq 2^{-5}\\
 \Delta \,\,{\rm dyadic}}} \Delta^{-\frac12}\sum_{K_0}\sum_{K\leq CK_0}K^s
  \|u_{K_0}\|_4 \| v_K\|_4  \\
  & \hspace{2cm}\cdot   \left\|\left( \sum_{\substack{2^{10}K_3^3<L\\
   \frac{\Delta}{2}K_0^2< L\leq \Delta K_0^2}}L|h_{K_0, jK_3, L}|^2\right)^{1/2}\right\|_2\,.
\end{aligned}
\end{equation}
From (\ref{L2hKL1}), (\ref{est-3-K3L-3}) is majorized by
\begin{equation}\label{est-4-K3L-3}
\begin{aligned}
&\,   C{\mathfrak M_1}\|\phi\|^2_{H^s} \|u\|^2_{Y_s}\sum_{K_3} \|u_{K_3}\|_\infty \sum_{\substack{\Delta\leq 2^{-5}\\
 \Delta \,\,{\rm dyadic}}} \Delta^{\frac12}\sum_{K_0}\sum_{K\leq CK_0}K^s
  \|u_{K_0}\|_4 \| v_K\|_4 \\
\leq & \, C{\mathfrak M_1}\|\phi\|^2_{H^s} \|u\|^3_{Y_s} 
     \left(\sum_{K_0}K_0^{2s}\|u_{K_0}\|_{0, \frac13}^2\right)^{1/2}
   \left(\sum_{K}\|v_{K}\|_{0, \frac13}^2\right)^{1/2} \\
\leq &\, C{\mathfrak M_1}\|\phi\|^2_{H^s} \|u\|^4_{Y_s}     \,.
\end{aligned}
\end{equation}
This finishes the proof for the case (\ref{K3L-3}).\\

For the contribution of (\ref{LK0-3}), we estimate (\ref{w1s2}) by
\begin{equation}\label{est-LK0-3-1}
   \sum_{K_1, K_2}\|u_{K_1}\|_\infty\|u_{K_2}\|_{\infty} \sum_{K_3}
 \|u_{K_3}\|_\infty
 \int \sum_{K_0} K_0 u^*_{K_0}\!\!\!\sum_{L > 2^{-5}K_0^2}|h_{K_0, jK_3, L}|
\sum_{K\leq CK_0}K^s{v}_K^* dxdt\,.
\end{equation}
By Cauchy-Schwarz inequality, (\ref{est-LK0-3-1}) is bounded by
\begin{equation}\label{est-LK0-3-2}
\sum_{K_1, K_2}\|u_{K_1}\|_\infty\|u_{K_2}\|_{\infty}\sum_{K_3} \|u_{K_3}\|_\infty
\sum_{K_0}\sum_{K\leq CK_0} K^s \int  v_K^* u^*_{K_0}\!
 \left(\sum_{L > 2^{-10}K_0^2}L|h_{K_0, jK_3, L}|^2\right)^{1/2}\!\!\!\!
 dxdt\,.
\end{equation}
Employing H\"older inequality with $L^4$ norms for the first two functions and 
$L^2$ norm for the last one, we dominate (\ref{est-LK0-3-2}) by
\begin{equation}\label{est-LK0-3-3}
\begin{aligned}
& C{\mathfrak M}_1\|u\|_{Y_s}^2\sum_{K_3} \|u_{K_3}\|_\infty
 \sum_{K_0}\sum_{K\leq CK_0} K^s\|u_{K_0}\|_4 \|v_K\|_4 
 \left\|\left( \sum_{L > 2^{-5}K_0^2}L|u_{ jK_3, L}|^2\right)^{1/2}\right\|_2 \\
\leq & \, C{\mathfrak M}_1\|u\|_{Y_s}^2\sum_{K_3} \|u_{K_3}\|_\infty
 \sum_{K_0}\sum_{K\leq CK_0} K^s\|u_{K_0}\|_{0,\frac13} \|v_K\|_{0, \frac13} 
 \|u\|_{0, \frac12} \\
\leq & \,   C{\mathfrak M}_1 \| \phi \|_{H^s}\|u\|_{Y_s}^4  \,.
\end{aligned}
\end{equation}
Hence we complete the case of (\ref{LK0-3}).\\

\end{document}